\documentclass[a4paper]{amsart}
\usepackage{amssymb} 
\usepackage[utf8]{inputenc}
\usepackage[T1]{fontenc} 
\usepackage{lmodern}
\usepackage[final]{microtype} 
\usepackage[hidelinks,pagebackref]{hyperref} 
   \hypersetup{final} 

\usepackage{todonotes}

\usepackage{mathtools} 
\usepackage[shortlabels]{enumitem}

\usepackage{tikz-cd} 
\usetikzlibrary{decorations.markings} 
\tikzset{negated/.style={
        decoration={markings,
            mark= at position 0.5 with {
                \node[transform shape] (tempnode) {$\backslash$};
            }
        },
        postaction={decorate}
    }
}

\usepackage{newunicodechar} 
\newunicodechar{ﬁ}{fi}
\newunicodechar{ﬀ}{ff}

\newtheorem{theorem}{Theorem}[section]
\newtheorem{lemma}[theorem]{Lemma}
\newtheorem{corollary}[theorem]{Corollary} 
\newtheorem{proposition}[theorem]{Proposition} 

\theoremstyle{definition}

\theoremstyle{remark}
\newtheorem{remark}[theorem]{Remark}

\numberwithin{equation}{section}

\newcommand*{\R}{\mathbb{R}}
\newcommand*{\N}{\mathbb{N}}

\newcommand{\dd}{\mathrm{d}}

\DeclareMathOperator{\E}{\mathbb{E}} 
\newcommand*{\p}{\mathbb{P}} 

\DeclareMathOperator{\Ent}{Ent}	
\DeclareMathOperator{\Var}{Var}	

\makeatletter
\newcommand{\addresseshere}{%
  \enddoc@text\let\enddoc@text\relax
}
\makeatother

\title[Thinning]{Thinning Operation via the Poisson-F\"ollmer Process }

\author[I. Kavvadias]{Ioannis Kavvadias}
\address{University of Warsaw, Institute of Mathematics, Banacha 2, 02--097 Warsaw, Poland.}
\email{i.kavvadias@mimuw.edu.pl}

\date{\today}

\makeatletter
\@namedef{subjclassname@2020}{\textup{2020} Mathematics Subject Classification}
\makeatother


\begin{document}

\maketitle

ABSTRACT. We give an alternative proof of Yu's Thinning Lemma and of the Law of Thin Numbers using a stochastic variational formula for the relative entropy. We show that this approach yields new convergence rates for the Law of Thin Numbers, complementing and extending previously known results.

\section{Introduction}

\subsection{Thinning}

The thinning operation was introduced by R\'enyi in \cite{R}, who used it to provide an alternative characterization of Poisson measures. R\'enyi’s thinning operation on a discrete random variable is a natural discrete analog of the scaling operation for continuous random variables. 
\newline

\textbf{Definition}: Given $a \in [0,1]$ and a discrete random variable $X$ taking values in $\N,$ the $a$--$thinning$ of $X$ is defined as
\begin{displaymath}
    T_{a}(X) := \sum_{i=1}^{X} B_{i}
\end{displaymath}
where $ (B_{i})_{i \ge 1} $ are independent and identically distributed Bernoulli random variables with parameter $a,$ that is $\p (B_{i} =1 ) = 1 - \p(B_{i} = 0) = a,$ independent of $X.$
\newline

We denote natural numbers by $\N := \{0, 1, . . .\}$ and real numbers by $\R.$ Moreover, whenever $X$ is an $\N$-valued random variable with law $\mu,$ we write $T_{a}(\mu) = T_{a}(X).$ For $T > 0,$ let $\pi_{T}$ denote the Poisson measure on $\N$ with parameter $T$, namely
\begin{displaymath}
    \pi_{T}(k) := e^{-T} \frac{T^{k}}{k !}, \quad \forall k \in \N,
\end{displaymath}
and by $\pi_{0}$ the point-mass at 0. One of the fundamental properties of the thinning operation is that it preserves a broad class of discrete random variables, among them the Poisson distribution. More precisely,
\begin{equation}
\label{eq:thinn_poiss}
       T_{a}(\pi_{T}) = \pi_{a T}.
\end{equation}
A classical result associated with thinning is the \textit{Law of Thin Numbers} (see Theorem \ref{thm:law}), introduced by Harremo\"es et al.\ in \cite{HJK1}, which may be viewed as a generalization of Poisson’s theorem. Roughly speaking, the theorem states that suitable thinning and convolution procedures yield convergence to the Poisson measure. Beyond the convergence itself, it is natural to seek quantitative versions of this phenomenon, measuring the rate at which convergence occurs. The study of such convergence rates has turned out to be closely related to modified log-Sobolev inequalities for the Poisson measure.

\subsection{Modified log-Sobolev Inequalities}
Poisson measures $\pi_{T}$ on $\N$ which are the discrete analogues of Gaussian measures, do not satisfy classical log-Sobolev inequalities. Rather, they satisfy modified log-Sobolev inequalities, introduced by Bobkov and Ledoux in \cite{BL} and the sharpest form of these
inequalities is due to Wu in \cite{W}. Let $f : \N \rightarrow (0, \infty)$ be an $\pi_{T}-$integrable function and define
\begin{displaymath}
    \Ent_{\pi_{T}} (f) := \sum_{k=0}^{\infty} f(k) \log f(k) \pi_{T}(k) - \Bigg( \sum_{k=0}^{\infty} f(k) \pi_{T}(k) \Bigg) \log \Bigg( \sum_{k=0}^{\infty} f(k) \pi_{T}(k) \Bigg).
\end{displaymath}
The modified log-Sobolev inequalities due to \cite{BL} and \cite{W} take the form
\begin{equation}
\label{eq:wu}
\Ent_{\pi_{T}}(f) \le T \E_{\pi_{T}} \Big[ \Psi(f,Df) \Big] \le
T \E_{\pi_{T}} \Bigg( \frac{(Df)^{2}}{f} \Bigg).
\end{equation}
where
\begin{displaymath}
    Df(k):= f(k+1) - f(k), \quad k \in \N
\end{displaymath}
is the discrete derivative of $f$ and 
\begin{equation}
    \label{eq:Psi}
    \Psi(u,v):=(u+v) \log(u+v) - u \log(u) - (\log(u)+1)v
\end{equation}
for all $u>0$ and $u+v>0.$ The second inequality in \eqref{eq:wu} is a consequence of the calculus estimate $\Psi(u,v) \le \frac{v^{2}}{u}.$ 
\newline

Recently, Aryan, Rivera and Shenfeld in \cite{ARS} employed a stochastic variational approach developed in \cite{KL} to provide an alternative proof of Wu’s inequality the same way Lehec proved Gross's inequality in \cite{L}. Moreover, they used this framework to obtain stability estimates for Wu’s inequality, in analogy with the corresponding program carried out in \cite{ELS}. In the Gaussian setting, there is a beautiful proof of the classical Shannon's inequality due to Lehec \cite{L}, who showed how to deduce the inequality from a stochastic representation formula for the relative entropy with respect to the Gaussian measure. The present work is motivated by this proof. Finally, we mention the works \cite{EM},\cite{EMZ}, who exploit this stochastic representation approach for the relative entropy with respect to the Gaussian measure to derive stability estimates for Shannon's inequality and convergence rates in the entropic Central Limit Theorem.

\subsection{Results}

Since our perspective is primarily information-theoretic, we introduce the following information-theoretic notation: Let $f : \mathbb{N} \to [0,\infty)$ satisfy $ \int_{\N} f \dd \pi_{T}=1 $ and $\mu := f \pi_T.$ Define
\begin{displaymath}
    H(\mu|\pi_{T}) := \sum_{k=0}^{\infty} f(k) \log \big( f(k) \big) \pi_{T}(k) = \int_{\N} f \log(f) \dd \pi_{T}
\end{displaymath}
to be the \textit{relative entropy} of a positive measure $\mu$ on $\N$ with respect to $\pi_{T}.$ We interpret $0 \log (0)$ by continuity, namely as 0. Note that 
\begin{displaymath}
    H(\mu | \pi_{T}) = \Ent_{\pi_{T}} (f).
\end{displaymath}
We will use the notation $H(X|\pi_{T}) = H(\mu | \pi_{T} )$ when $X$ is a discrete random variable taking values in $\N$ with $X \sim \mu .$ Our starting point is the following theorem of Yu. 
\begin{theorem} [Yu \cite{Y}]
    \label{thm:Yu}
    Let $X$ be a random variable with nonnegative density $f$ with respect to $\pi_{T}.$ Then, for all $a \in$ $[0,1],$ we have 
\begin{displaymath}
    H(T_{a}(X)|\pi_{aT}) \le a H(X|\pi_{T}).
\end{displaymath}
\end{theorem}
The above is called the Thinning Lemma and plays a central role in many information theoretic inequalities for discrete random variables (see \cite{J}).
\newline

We provide a new proof of the above inequality as well as an explicit variational formula for the relative entropy on the left hand side. As a by-product of our proof, we will obtain an improvement of the Thinning Lemma under convexity assumptions, namely under the assumption that $f$ is \textit{ultra log-concave}. Recall that a positive function $f: \N \rightarrow (0, \infty)$ is \textbf{ultra log-concave} if
\begin{displaymath}
    f(k)^{2} \ge \frac{k+1}{k}f (k + 1)f (k - 1), \quad \forall k \ge 1.
\end{displaymath}
Denote by 
\begin{displaymath}
    \delta_{\mathrm{Thinn}} (f) := a H(X|\pi_{T}) - H(T_{a}(X)| \pi_{aT} )
\end{displaymath}
where $f$ is the density function of $X$ with respect to $\pi_{T},$ that is $X \sim f \dd \pi_{T}.$ We will frequently work with the function $\phi \colon [0,\infty)\to [0,\infty)$  given by 
\begin{equation}
    \label{eq:phi}
    \phi(x) = x\log x - x + 1.
\end{equation} 
Our first result is the following theorem.
\begin{theorem}
\label{thm:main1.2}
     Fix $T>0.$ Let $X$ be a random variable whose density $f : \N \rightarrow (0, \infty) $ with respect to $\pi_{T}$ is ultra log-concave, then we have
    \begin{displaymath}
        \delta_{\mathrm{Thinn}}(f) \ge a(1-a) \frac{T^{2}}{2} \Theta_{\frac{f(0)}{f(1)}} \Bigg( \frac{\E X}{T} \Bigg)
    \end{displaymath}
    where, for $c>0,$ we denote
    \begin{displaymath}
        \Theta_{c}(z):= \frac{z^{2}}{1+cz} \log \Big( \frac{1}{1+cz} \Big) - \frac{z^{2}}{1+cz} + z^{2}= z^{2} \phi \Bigg( \frac{1}{1 + cz} \Bigg), \quad z\ge 0.
    \end{displaymath}
\end{theorem} 

We now introduce the \textit{Fisher information}. Let $f:\N \rightarrow [0,\infty)$ and $ \mu := f \pi_{T}.$ Define
\begin{equation}
    \label{eq:FI}
    I(\mu | \pi_{T}):= T\sum_{k=0}^{\infty} \frac{(Df(k))^{2}}{f(k)} \pi_{T}(k) = T \int_{\N} \frac{(Df)^{2}}{f} \dd \pi_{T}.
\end{equation}
Throughout, we adopt the convention $\frac{0}{0}=0.$ As before, we will use the notation $I(X | \pi_{T}) = I (\mu | \pi_{T})$ when $X$ is a discrete random variable with $X \sim \mu.$ 

\begin{remark}
\label{rem:definitions}
Let us compare our definition of Fisher Infromation \eqref{eq:FI} with the one appearing in the information-theoretic literature (see \cite{J},\cite{HJK3}). The \textit{scaled score function} of a discrete random variable taking values in $\N$ with $T=\E X$ and $X \sim \rho$ is defined as
\begin{displaymath}
    \rho_{X} (k) : = \frac{(k+1) \rho (k + 1) }{T \rho(k) } - 1, \quad \forall k\in \N
\end{displaymath}
and the \textit{scaled Fisher Information} is defined as
\begin{equation}
    \label{eq:scaledFI} 
    K(X) := T \E \big[ \rho_{X}(X)^{2} \big].
\end{equation}
By working with density functions, when $X \sim f \dd \pi_{T}$ and $T=\E X,$ an easy calculation shows that the definitions \eqref{eq:FI},\eqref{eq:scaledFI} coincide
\begin{displaymath}
    K \big( X \big) = I \big( X |\pi_{T} \big).
\end{displaymath}   
However, we choose to work with \eqref{eq:FI}, as several of the results discussed below do not require the assumption $T=\E X.$
\end{remark}
Motivated by the Law of Thin Numbers, we next investigate the rate at which $I(T_{a}(X) | \pi_{aT})$ tends to 0 when $a \rightarrow 0^{+}.$ Such bounds have already been studied in \cite{HJK2}, among other references. In particular, Harrem\"oes, Johnson and Kontoyiannis show that, if $\Var(X)$ is finite and $T = \E X$, then, for $n \ge 2$
\begin{equation}
    \label{eq:general_conc}
    H \Big( \sum_{i=1}^{n} T_{1/n}(X_{i}) |\pi_{T} \Big) \le \frac{ \Var(X)}{nT} + \frac{1}{n^{2}},
\end{equation}
where $\Var(X)$ denotes the variance of $X;$ see \cite[Proposition 19]{HJK2} and, in particular, \cite[Formula 30]{HJK2}. The authors then restrict their attention to the class of \textit{ultra log-concave} measures. To introduce this notion, we first recall that a positive function $f:\N \rightarrow (0,\infty)$ is \textbf{log-concave} if
\begin{displaymath}
    f(k)^{2} \ge f (k + 1)f (k - 1) \quad \forall k \ge 1.
\end{displaymath}
Harrem\"oes, Johnson and Kontoyiannis show that if $X$ has \textit{ultra log-concave} distribution, that is
$X \sim f \dd \pi_{T}$ with $f$ positive and log-concave, then the following estimate holds true 
\begin{equation}
\label{eq:log_conc}
     K( T_{a}(X)  ) \le a^{2} K( X ) 
\end{equation}
$\forall a\in [0,1].$ Our first estimate in that direction complements the above result in the general case.
\begin{theorem}
    \label{thm:FIsc}
    Let $X$ be a random variable with nonnegative density $f$ with respect to $\pi_{T}.$ Then, for all $a \in$ $[0,1],$ we have 
\begin{displaymath}
    I( T_{a}(X) | \pi_{aT} ) \le a I(X| \pi_{T}).
\end{displaymath}
\end{theorem}
As before, we introduce the following deficit
\begin{displaymath}
     \label{thm:FIstab}
    \delta_{\mathrm{FI}}(f) :=  a I( X | \pi_{T} ) - I(T_{a}(X) | \pi_{aT})
\end{displaymath}
and exploit stability estimates of the above inequality under a convexity assumption on the density function $f.$
\begin{theorem}
\label{thm:main1.4}
     Fix $T>0.$ Let $X$ be a random variable whose density $f : \N \rightarrow (0, \infty) $ with respect to $\pi_{T}$ is ultra log-concave, then we have
\begin{displaymath}
    \delta_{\mathrm{FI}}(f) \ge a(1-a)T^{2} \Delta_{\frac{f(0)}{f(1)}} \Bigg( \frac{\E X}{T} \Bigg)
\end{displaymath}
where, for $c>0,$ we denote
\begin{displaymath}
    \Delta_{c}(z):= z^{5} \Big( \frac{c}{1 + cz} \Big)^{2}  , \quad z\ge 0.
\end{displaymath}
\end{theorem}
\begin{remark}
    We postpone the comparison between Theorem \ref{thm:FIsc} and the result \eqref{eq:general_conc}, as well as that between Theorem \ref{thm:main1.4} and the result \eqref{eq:log_conc} of Harrem\"oes, Johnson and Kontoyiannis, to Subsection \ref{sss:convrates}, Remark \ref{rem:general_comparison} and Subsection \ref{sss:Fisher_information}, Remark \ref{remark:comparison}, respectively.
\end{remark}
\begin{remark}
    As observed by Aryan, Rivera and Shenfeld, the parameters $\E X$ and $f(0), f(1)$ appear naturally in the context of stability estimates for Wu's inequality. We refer to \cite[Remark 1.6]{RS} for a related discussion concerning discussion these parameters in the context of modified log-Sobolev inequalities.
\end{remark}

The comparison inequalities derived above for the Fisher information of thinned random variables are motivated by the study of convergence rates in the the Law of Thin Numbers: 
\begin{theorem} 
      [Harrem\"oes, Johnson and Kontoyiannis \cite{HJK1}] 
If $T=\E X$ and $H(X|\pi_{T}) < \infty,$ then 
\begin{displaymath}
\label{thm:law}
    H \Big( \sum_{i=1}^{n} T_{1/n}(X_{i}) \big| \pi_{T} \Big) \longrightarrow 0 
\end{displaymath}
as $n\rightarrow \infty,$ where $(X_{i})_{i \ge 1}$ iid copies of $X.$ 
\end{theorem}
In Subsection \ref{sss:convrates}, we provide a short proof of the Law of Thin Numbers, based on the Poisson--F\"ollmer process. Moreover, Harrem\"oes, Johnson and Kontoyiannis proved in \cite{HJK2} that, if $X $ has an \textit{ultra bounded} distribution, then the following asymptotic estimate holds
\begin{equation}
    \label{eq:asympt}
     \limsup_{ n \rightarrow \infty } \Bigg[ n^{2} H \Big( \sum_{i=1}^{n} T_{1/n} ( X_{i} ) \big| \pi_{T} \Big) \Bigg] \le  \Big( \frac{\Var(X)}{T } - 1 \Big)^{2}
\end{equation}
with $T = \E X,$ provided that the numerator does not vanish. The estimate \eqref{eq:asympt} is formulated for \textit{ultra bounded} distributions, a class that contains the class of \textit{ultra log-concave} distributions; see \cite[Defintion 5]{HJK2} and \cite[Proposition 9]{HJK2}. In fact, it is a consequence of the more general \cite[Corollary 3.2]{HJK2}. In Proposition \ref{prop:asymptotic}, we recover the same asymptotic estimate on the left-hand side of \eqref{eq:asympt} without assuming ultra boundedness (or ultra log-concavity).

\begin{remark}
More recently, while this work was being finalized, Shenfeld, Baptista and Peluchetti published a preprint \cite{SBP} which exploits the same stochastic construction in the context of machine learning applications. 
\end{remark}

\subsection{Organization}
Section \ref{ss:2} introduces the Poisson--F\"ollmer process and its main properties. Section \ref{ss:3} contains our principal results and comparisons with earlier work. Subsection \ref{sss:thinning} establishes a stability estimate of the Thinning Lemma under a convexity assumption on the density function, while Subsection \ref{sss:Fisher_information} studies the convergence of the Fisher information of $a$--thinning, both in general case and under convexity. In Subsection \ref{sss:convrates}, we derive finite $n$ bounds on the rate of convergence for the Law of Thin Numbers for general, log-concave, and ultra log-concave density functions, together with a new asymptotic estimate. Subsection \ref{sss:equality_cases} characterizes the equality cases in Yu’s Thinning Lemma (Theorem \ref{thm:Yu}) and Theorem \ref{thm:FIsc}. Finally, Subsection \ref{sss:size_bias} further develops connections with the existing literature.

\subsection{Acknowledgments} 
I would like to thank Joseph Lehec for introducing me to the topic of F\"ollmer process and for the useful discussions. I am very grateful to Rados{\l}aw Adamczak for carefully reading earlier versions of this manuscript and for his valuable comments and suggestions which have improved its quality.

\section{Preliminaries}

\label{ss:2}

In this section, we define the  Poisson--F\"ollmer process and present its fundamental properties. We begin with general preliminaries on the Poisson semigroup

\subsection{Preliminaries}
We say that a function $f : \N \rightarrow \R $ is $\pi_{t}-$integrable if
\begin{displaymath}
    \int_{\N} | f | \dd \pi_{t} = \sum_{n \in \N} | f(n) | \pi_{t}(n) < \infty.
\end{displaymath}
For fixed $t>0,$ let $f: \N \rightarrow \R$ be a $\pi_{t}-$integrable function, and we define
\begin{equation}
      \label{eq:semigroup}
    P_{t}f(k) := \sum_{n \in \N} f(k+n) \pi_{t} (n), \quad \forall k \in \N
\end{equation}
and for $t=0$ we set $P_{0}f=f.$ The Poisson semigroup is the family of operators $(P_{t})_{t \ge 0},$ satisfying the evolution equation
\begin{equation}
     \label{eq:heateq}
    \partial_{t}  P_{t}f(k)  = D  P_{t}f(k) = P_{t} D f(k) , \quad \forall k \in \N.
\end{equation}
Fix $T>0.$ Let $f: \N \rightarrow (0, \infty)$ be $\pi_{T-t}-$integrable and introduce the following functions:
\begin{equation}
     \label{eq:functF} 
    F(t,k):= \log P_{T-t}f(k), \quad \forall k \in \N
\end{equation}
and 
\begin{equation}
    \label{eq:functG}
    G(t,k):= e^{DF(t,k)}, \quad \forall k \in \N.
\end{equation}
We shall make use of the following lemma from \cite[Lemma 2.1]{ARS}.
\begin{lemma}
    We have
    \begin{equation}
         \label{eq:form1}
        \partial_{t} F(t,k) = - e^{DF(t,k)} + 1, \quad \forall k \in \N
    \end{equation}
    and 
    \begin{equation}
         \label{eq:form2}
        \partial_{t} G(t,k) = - G(t,k) DG(t,k), \quad \forall k\in \N.
    \end{equation}
\end{lemma}

\subsection{The Poisson--F\"ollmer Process}

We now turn to the construction and properties of the
Poisson--F\"ollmer process. Fix $T > 0$, and let $\dd \mu := f \dd \pi_{T}$ be a probability measure on $\N.$ Building on the work of Budhiraja, Dupuis, and Maroulas in \cite{BDM}, Klartag and Lehec \cite{KL} constructed a stochastic counting process $(X_{t})_{t \in [0,T]}$ satisfying $X_{T} \sim \mu.$ We will briefly describe the process and refer to \cite{KL}, \cite{RS} for a complete description.

\medskip

Let $(\Omega, \mathcal{F}, \p )$ be the underlying probability space on which all random variables are defined. Let $N$ be a Poisson point process on $[0, T ] \times [0,\infty) \subset \R^{2}$ with intensity measure given by the Lebesgue measure. Thus, for every Borel set $F \subset [0,T] \times [0,\infty),$ the random variable $N(F)$ is Poisson distributed with parameter equal to the Lebesgue measure of $F.$ For each $t \in [0,T],$ define the sigma-algebra
\begin{displaymath}
    \mathcal{F}_{t} : = \sigma \big( \{ N (F ) : F \subset [0, t] \times [0, \infty)  \text{ is a Borel set} \} \big).
\end{displaymath}
The collection $(\mathcal{F}_{t})_{t \in [0,T]}$ forms a filtration. A stochastic process $(\lambda_{t})_{t \in [0,T ]}$, where $\lambda_{t} : \Omega \rightarrow \R$, is said to be predictable if the map $(t, \omega) \mapsto \lambda_{t}(\omega)$ is measurable with respect to the sigma algebra $\sigma ( \{ (s, t] \times B : s \le t \le T, B \in \mathcal{F}_{s} \} ).$ Let $(\lambda_{t})_{t \in [0,T ]}$ be a bounded, nonnegative, predictable process. We then define the associated counting process by
\begin{displaymath}
    X_{t}^{\lambda} (\omega) := N \big( \{(s, x) \in [0, T ] \times [0, \infty) : s < t, x \le \lambda_{s}(\omega) \} \big) .
\end{displaymath}
The process $(X_{t}^{\lambda})_{t \in [0,T]}$ is integer-valued, adapted to the filtration $(\mathcal{F}_{t})_{t \in [0,T]},$ and has left-continuous sample paths. Within this framework, Klartag and Lehec stated in \cite[Section 4]{KL}; see in particular \cite[Section 4, Remark 3]{KL}, the following stochastic representation formula for the relative entropy with respect to the Poisson measure:
\begin{equation}
    \label{eq:variational}
    H \big( \mu | \pi_{T} \big) = \inf_{\lambda_t} \Bigg\{ \E \int_{0}^{T} \phi(\lambda_{t}) \dd t   \Bigg\}
\end{equation}
where $\phi$ is given by \eqref{eq:phi} and the infimum runs over all nonnegative, bounded, predictable processes $(\lambda_{t})_{0 \le t \le T}$ such that $X_{T}^{\lambda} \sim \mu$. Moreover, they showed that there exists $\lambda$ such that for almost every $t$ and with probability one
\begin{equation}
     \label{eq:minimizer}
    \lambda_{t} = \frac{P_{T-t}f(X_{t}^{\lambda} + 1) }{P_{T-t}f(X_{t}^{\lambda}) }
\end{equation}
where $P_{t}$ is the Poisson semigroup, the resulting counting process $(X_{t})_{t \in [0,T]} := (X_{t}^{\lambda})_{t\in [0,T]}$ is well defined and equality is attained in \eqref{eq:variational} for this particular intensity. This last assertion is implicit in the proof of \cite[Theorem 4.2]{KL}. Finally, L\'opez-Rivera and Shenfeld show in \cite[Lemma 3.2]{RS} that the resulting counting process for $\lambda$ as in \eqref{eq:minimizer} also satisfies $X_{T}^{\lambda} \sim \mu,$ and is therefore admissible in the variational characterization \eqref{eq:variational}. Combined with \cite[Theorem 4.2]{KL}, this yields the representation formula \eqref{eq:variational}. 

\medskip

The process is named \textbf{Poisson--F\"ollmer process}, since it is the discrete analogue of the F\"ollmer process in the continuous setting, see \cite{F1},\cite{F2} and \cite{L}. Before moving on to the properties of the process, we state a lemma due to Klartag and Lehec that guarantees the existence of the process defined above. 

\begin{lemma}
     [Klartag, Lehec \cite{KL}, Lemma 4.3]
    \label{lem:uniq}
    Let $ G : [0, T ] \times \N \rightarrow (0, \infty)$ and assume that $G$ is continuous in the first variable and bounded. Then there exists a predictable, bounded, non-negative process $(\lambda_{t})_{0 \le t \le T}$ such that
    \begin{displaymath}
        \lambda_{t}:= G(t,X_{t}^{\lambda}),
    \end{displaymath}
    for almost every $t \le T$ and with probability one.
\end{lemma}
\begin{remark}
    \label{remark:function_G}
    Note that $\lambda_{t} = G(t,X_{t}),$ where $G$ is given by \eqref{eq:functG}.
\end{remark}

\begin{remark}
     \label{remark:existence}
    The proof of the Lemma \ref{lem:uniq} provides also uniqueness of the process $(\lambda_{t})_{0 \le t \le T}$ via a fixed-point argument, which requires the function $f$ to be bounded away from zero and infinity. This can also be guaranteed if $f$ is log-concave; see \cite[Corollary 2.4]{RS}.
\end{remark}

\begin{remark}
     \label{remark:predictability}
     Let $\lambda, \tilde{\lambda}$ be predictable processes such that $\lambda_{t} = \tilde{\lambda}_{t}$ for a.e. $t \in [0,T], \p$-- a.s. Then, the associated processes $X^{\lambda}$ and $X^{\tilde{\lambda}}$ are indistinguishable. Indeed, since $N$ is a Poisson point process with intensity given by Lebesgue measure on $[0, T ] \times [0,\infty)$, we obtain
     \begin{displaymath}
         \E N \Big( \{ ( t,s ) : \min (\lambda_{t} ,\tilde{\lambda}_{t} ) < s \le \max (\lambda_{t} ,\tilde{\lambda}_{t} ) \} \Big) = \E \int_{0}^{T} \big| \lambda_{t} - \tilde{\lambda}_{t} \big| \dd t = 0 ,
     \end{displaymath}
     where we used predictability. It follows that, $\p$-- a.s., $X_{t}^{\lambda} = X_{t}^{\tilde{\lambda}}$ for all $t \in [0,T],$ so $X^{\lambda}$ and $X^{\tilde{\lambda}}$ are indistinguishable. We may define $\tilde{\lambda}$ by the right-hand side of equation \eqref{eq:minimizer}, which is predictable since it is left continuous and the resulting processes $X^{\tilde{\lambda}}$ is indistinguishable from the Poisson--F\"ollmer process. From now on, we will be working with $\tilde{\lambda},$ which, for simplicity of notation, we denote by $\lambda.$ In particular, equation \eqref{eq:minimizer} holds with probability one for all $t.$ As a consequence, $\lambda$ is with probability one left continuous and continuous except for finitely many points.
    
\end{remark}

\subsection{Properties of the Poisson--F\"ollmer Process}

Let us establish some useful properties of the above construction. 

\begin{lemma}
     \label{lem:process}
    Let $(X_{t})_{0 \le t \le T}$ be the Poisson F\"ollmer process. Then, we have
    \begin{enumerate}
        \item $X_{t} \sim P_{T-t}f \dd \pi_{t}$ for all $0 \le t \le T.$
        \item Denote the compensated process $(\tilde{X_{t}})_{0 \le t \le T}$ as 
        \begin{displaymath}
            \tilde{X_{t}} := X_{t} - \int_{0}^{t} \lambda_{s} \dd s.
        \end{displaymath}
        Then, $(\tilde{X_{t}})_{0 \le t \le T}$ is a martingale.
    \end{enumerate}
\end{lemma}
Item (1) can be found in \cite{RS}, Lemma 3.2 while item (2) can be found in \cite{KL}, Lemma 4.1.

\begin{lemma}
     [L\'opez-Rivera, Shenfeld \cite{RS}, Lemma 3.3]
     \label{lem:martingale}
    The process $(\lambda_{t})_{0 \le t \le T}$ is a martingale with common expectation
    \begin{displaymath}
        \E( \lambda_{t} ) = P_{T} f(1) = \frac{\E X}{T}.
    \end{displaymath}
\end{lemma}
The following result is implicit in the proof of \cite[Theorem 1.2, estimate (3.12)]{ARS} and shows how ultra log-concavity allows us to upper bound the second derivative of $F(s,X_{s})$. To state this result we first need to introduce
\begin{displaymath}
    D^{2}:= D \circ D.
\end{displaymath}

\begin{proposition}
     [Aryan, L\'opez--Rivera and Shenfeld \cite{ARS}]
    \label{proposition:ARS}
    Let $T>0$ and $f:\N \rightarrow (0, \infty)$ be  a $\pi_{T}-$integrable ultra log-concave function, such that $ \mu = f  \pi_{T}$ is a probability measure. Then, we have
    \begin{displaymath}
        e^{D^{2}F (s,X_{s})} \le \frac{1}{1 + \frac{f(0)}{f(1)} \lambda_{s}} < 1.
    \end{displaymath}   
\end{proposition}
Recall the functions $\Psi$ and $\phi$ introduced in \eqref{eq:Psi} and \eqref{eq:phi}, respectively. We will also make use of the identity 
\begin{equation}
     \label{eq:calculation}
    \Psi(u,v)= \phi \Big( \frac{u+v}{u} \Big)u
\end{equation}
which holds for all $u>0$ and $u+v>0.$

\section{Results}

\label{ss:3}

We begin with Subsection \ref{sss:thinning}, where we establish a new proof of Yu’s inequality (Theorem \ref{thm:Yu}) via the Poisson--F\"ollmer process and derive the improved estimate of Theorem \ref{thm:main1.2} assuming convexity of the density function.

\subsection{Thinning Lemma and applications}

\label{sss:thinning}

\begin{lemma}
     \label{lem:law}
    Let $X $ have density $f$ with respect to $\pi_{T}.$ Then, $T_{a}(X) \sim P_{T(1-a)} f \dd \pi_{aT}$ for all $a$ in $[0,1].$ 
\end{lemma}
\begin{proof}
    For $z \in \N,$ we compute
\begin{align*}
\p (T_{a}(X)=z)
&= \sum_{x=z}^{\infty} \p(X=x)\p(T_{a}(X)=z \mid X=x) \\
&= \sum_{x=z}^{\infty} f(x) T^{x} e^{-T}
    \frac{1}{z!(x-z)!}a^{z}(1-a)^{x-z} \\
&= \frac{(aT)^{z}}{z!} e^{-aT}
    \sum_{x=z}^{\infty}
    f(x)\frac{(T(1-a))^{x-z}}{(x-z)!}
    e^{-T(1-a)} \\
&= \frac{(aT)^{z}}{z!} e^{-aT}
    \sum_{x=0}^{\infty}
    f(x+z)\frac{(T(1-a))^{x}}{x!}
    e^{-T(1-a)} \\
&= P_{T(1-a)}f(z)\,\pi_{aT}(z).
\end{align*}
Observe that when $a=0,$ the random variable $T_{a}(X) = 0 $ with probability $1.$ This agrees with the measure $P_{T}f \pi_{0},$ as $\pi_{0}$ is the Dirac mass at $0$ and $P_{T}f(0)=1.$ 
\end{proof}
\begin{remark}
    Applying the variational formula \eqref{eq:variational} to the relative entropy of the thinned random variable $T_{a}(X)$ with respect to the Poisson measure $\pi_{aT},$ we obtain
    \begin{equation}
        \label{variational_thin}
        H \big( T_{a}(X) | \pi_{aT} \big) = \inf_{\lambda_t} \Bigg\{ \E \int_{0}^{aT} \phi(\lambda_{t}) \dd t   \Bigg\}
    \end{equation}
    where infimum runs over all nonnegative, bounded, predictable processes $(\lambda_{t})_{0 \le t \le a T}$ such that $X_{aT}^{\lambda} \sim T_{a}(X)$.
\end{remark}

\begin{proposition}
     \label{prosotion:3.3}
    Let $X$ be a random variable with density $f$ with respect to $\pi_{T},$ and let $a \in [0,1].$ Assume that $f$ is bounded away from zero and infinity. Then $(\lambda_{t})_{0 \le t \le aT}$ defined in \eqref{eq:minimizer}, is the minimizing process for $T_{a}(X)$ with respect to the Poisson measure $\pi_{aT}.$ Therefore,
    \begin{displaymath}
        H(T_{a}(X) | \pi_{aT} ) = \E \int_{0}^{aT} \phi (\lambda_{t} ) \dd t.
    \end{displaymath}
\end{proposition}
\begin{proof}
The case $a=0$ is trivial, while the case $a=1$ follows from \eqref{eq:variational} and \eqref{eq:minimizer}. Let $X \sim f \dd \pi_{T}$ and $a \in (0,1).$ 
Denote by $\lambda_{t}^{a}, X_{t}^{\lambda^{a}}$ the optimal stochastic intensity for $T_{a}(X)$ and its corresponding counting process. In view of Lemmas \ref{lem:law} and \ref{lem:uniq}, the process $\lambda^{a}$ is the unique process that satisfies 
\begin{equation}
     \label{eq:eq_for_minim}
    \lambda_{t}^{a} = \frac{P_{aT-t} P_{T(1-a)}f(X_{t}^{\lambda^{a}}+1)}{P_{aT-t} P_{T(1-a)}f(X_{t}^{\lambda^{a}})}=\frac{ P_{T-t}f(X_{t}^{\lambda^{a}}+1)}{ P_{T-t}f(X_{t}^{\lambda^{a}})}, \quad \forall 0\le t \le aT.
\end{equation}
By Lemma \ref{lem:uniq} and Remark \ref{remark:function_G}, $(\lambda_{t})_{0 \le t \le T}$ is the unique process satisfying \eqref{eq:eq_for_minim} on $[0,T].$ Thus, $\lambda_{t}^{a} = \lambda_{t}$ for all $0 \le t \le aT.$ From Lemmas \ref{lem:process} and \ref{lem:law}, we obtain $X_{aT}^{\lambda^{a}} = X_{aT}
\overset{d}{=} T_{a}(X),$ which allows us to conclude.
\end{proof}
\begin{remark}
    Under the assumptions that $f$ is log-concave, Proposition \ref{prosotion:3.3} remains valid; see Remark \ref{remark:existence}.
\end{remark}
We can now provide a simple proof of Yu’s Thinning Lemma \ref{thm:Yu}.
\begin{proof}
The cases $a=0$ and $a=1$ are trivial. We therefore assume $a \in (0,1).$ Let us first assume $X \sim f \dd \pi_{T}$ with $f$ separated from zero and infinity. Then, since the process $(\lambda_{t})_{t}$ is a martingale by Lemma \ref{lem:martingale}, and $\phi$ is convex, it follows that $(\phi(\lambda_{t}))_{t}$ is a submartingale. Hence, by Proposition \ref{prosotion:3.3},
\begin{displaymath}
    H(T_{a}(X)|\pi_{aT})=\E \int_{0}^{aT} \phi(\lambda_{t}) \dd t = a \E \int_{0}^{T} \phi(\lambda_{as}) \dd s \le a \E \int_{0}^{T} \phi(\lambda_{s}) \dd s = a  H(X|\pi_{T}).
\end{displaymath}
In the general case, we consider truncated densities
\begin{equation}
     \label{eq:densities}
    f_{k}:= \min \Big \{ k , \max \Big\{ f , \frac{1}{k} \Big\} \Big\}
\end{equation}
and $X_{k} \sim \frac{f_{k} \pi_{T}}{ \int_{\N} f_{k} \dd \pi_{T}}.$ Since $f_{k} \rightarrow f$ pointwise and $0 \le f_{k} \le f +1 ,$ with $f$ integrable with respect to $\pi_{T},$ the dominated convergence theorem implies 
\begin{displaymath}
    \int_{\N} f_{k} \dd \pi_{T} \longrightarrow \int_{N} f \dd \pi_{T}.
\end{displaymath}
By Lemma \ref{lem:law}, we obtain $T_{a}(X_{k}) \sim \frac{P_{T(1-a)}f_{k} \pi_{aT}}{\int_{N} f_{k} \dd \pi_{T}}.$ Note that the normalising constant is consistent with the semigroup property of \eqref{eq:semigroup}. Similarly, as above, an application of the dominated convergence theorem yields $P_{T(1-a)}f_{k}  \rightarrow P_{T(1-a)} f.$ Thus, $T_{a}(X_{k})$ converges weakly to $T_{a}(X).$ Assume without loss of generality that $H(X | \pi_{T}) < \infty$ and by lower semi-continuity of relative entropy, we conclude
\begin{align*}
     H \big( T_{a}(X) | \pi_{aT} \big)
     &\le \liminf_{k \rightarrow \infty}
      H \big( T_{a}(X_{k}) | \pi_{aT} \big) \\
      &\le a\liminf_{k \rightarrow \infty}
       H \big( X_{k} | \pi_{T} \big)
       = a H \big( X | \pi_{T} \big)
\end{align*}
where the last limit follows again by dominated convergence.
\end{proof}

The following lemma has already appeared in \cite[Lemma 3.3]{ARS}. Since it appears naturally in several arguments below, we provide its proof for completeness. 

\begin{lemma}
     \label{lem:Ito}
    Let $G$ be as in \eqref{eq:functG} and $h:\R \rightarrow \R$ differentiable function. Then, we have
    \begin{displaymath}
        \E \big[ h(\lambda_{t}) \big] -  h(\lambda_{0})  =\int_{0}^{t} \E \Big(  \lambda_{s} \big[ h(\lambda_{s} + DG(s,X_{s})) - h (\lambda_{s}) - h'(\lambda_{s})DG(s,X_{s}) \big]  \Big)   \dd s
    \end{displaymath}
\end{lemma}

\begin{proof}
The function $h \circ G (\cdot, k): [0,T] \rightarrow \R $ is continuous for all $k \in \N$ and the function $t \mapsto (h \circ G)(s, X_{s})$ is a piecewise absolutely continuous function in $s.$ Taking into account that $X_{s}$ is a sum of finitely many jumps, we get
\begin{align*}
h(\lambda_{t}) - h(\lambda_{0})
&= \big( h \circ G \big)(t,X_{t})
   - \big( h \circ G \big)(0,X_{0}) \\[6pt]
&= \int_{0}^{t}
    \partial_{s}(h \circ G)(s,X_{s}) \dd s
   + \int_{0}^{t}
    D(h \circ G)(s,X_{s}) \dd X_{s} \\[6pt]
&= \int_{0}^{t}
    h'\big(G(s,X_{s})\big)
    \partial_{s}\big(G(s,X_{s})\big)\dd s \\
&\quad
   + \int_{0}^{t}
    \Big[
        h\big(G(s,X_{s} + 1)\big)
        - h\big(G(s,X_{s})\big)
    \Big]
    \dd X_{s}^{\lambda} \\[6pt]
    &\overset{\eqref{eq:form2}}{=} \int_{0}^{t}
    \Big[
        -h'(\lambda_{s})\lambda_{s}
        D\big(G(s,X_{s})\big)
    \Big]\dd s \\
&\quad
   + \int_{0}^{t}
    \Big[
        h\big(\lambda_{s}+DG(s,X_{s})\big)
        - h(\lambda_{s})
    \Big]
    \dd X_{s} \\[6pt]
&= \int_{0}^{t}
    \lambda_{s}
    \Big[
        h\big(\lambda_{s} + DG(s,X_{s})\big)
        - h(\lambda_{s})
        - h'(\lambda_{s})DG(s,X_{s})
    \Big]\dd s \\
&\quad
   + \int_{0}^{t}
    \Big[
        h\big(\lambda_{s}+DG(s,X_{s})\big)
        - h(\lambda_{s})
    \Big]
    \dd \tilde{X}_{s}.
\end{align*}
The conclusion follows by taking expectations, since the compensated process is a martingale (see Lemma \ref{lem:process}).    
\end{proof}
We are now ready for the proof of Theorem~\ref{thm:main1.2}.
\begin{proof}
Applying Proposition \ref{prosotion:3.3}, the deficit becomes
\begin{align*}
\delta_{\mathrm{Thinn}}(f)
&:= a H(\mu | \pi_T) - H(T_a(\mu) | \pi_{aT}) \\
&= a \Big( \int_{0}^{T} \E \phi(\lambda_t)\,\dd t - \int_{0}^{T} \E \phi(\lambda_{at})\, \dd t \Big) \\
&= a \int_{0}^{T} \int_{at}^{t}
\partial_s \big( \E \phi(\lambda_s) \big)\, \dd s\, \dd t.
\end{align*}
A calculation shows that 
\begin{displaymath}
    \phi (y) - \phi(x) - \phi'(x) (y-x) = x \phi \Big( \frac{x}{y} \Big), \quad \forall x,y>0.
\end{displaymath}
Under the ultra log-concavity assumption, the functions $s \mapsto \lambda_{s} $ and $s \mapsto DG(s,X_{s})$ are bounded; see \cite[Corollary 2.4]{RS}. Moreover, they are continuous except at finitely many points. Applying Lemma \ref{lem:Ito} with $h(x)=\phi(x),$ it follows that the integrand on the right-hand side is continuous. Hence,
\begin{displaymath}
    \partial_s \big( \E \phi(\lambda_s) \big) = \E \Bigg[ \lambda_{s}^{2} \phi \Bigg( \frac{\lambda_{s} + DG(s,X_{s})}{\lambda_{s}} \Bigg) \Bigg].
\end{displaymath}
Next, note that
\begin{displaymath}
    DG(t,k) = \frac{P_{T-t}f(k+2)}{P_{T-t}f(k+1)} - \frac{P_{T-t}f(k+1)}{P_{T-t}f(k)} = \frac{P_{T-t}f(k+2)}{P_{T-t}f(k+1)} -G(t,k)
\end{displaymath}
so 
\begin{equation}
     \label{eq:calc}
    \frac{G(t,k) + DG(t,k)}{G(t,k)} = \frac{P_{T-t}f(k+2) P_{T-t}f(k)}{P_{T-t}f(k+1)^{2}} = e^{D^{2} \log P_{T-t}f(k)}.
\end{equation}
Recall that $\lambda_{s}=DG(s,X_{s}).$ Combining the above with Proposition \ref{proposition:ARS}, we obtain
\begin{align*}
\delta_{\mathrm{Thinn}}(f)
&= a \int_{0}^{T} \int_{at}^{t}
\partial_s \big( \E \phi(\lambda_s) \big)\, \dd s\, \dd t \\
&= a \int_{0}^{T} \int_{at}^{t}
\E \Bigg[
\lambda_s^{2}
\phi \Big(
e^{D^{2} \log P_{T-s}f(X_s)}
\Big)
\Bigg] \dd s\, \dd t \\
&\ge
a \int_{0}^{T} \int_{at}^{t}
\E \Bigg[
\lambda_s^{2}
\phi \Bigg(
\frac{1}{1 + \frac{f(0)}{f(1)} \lambda_s}
\Bigg)
\Bigg] \dd s\, \dd t \\
&=
a \int_{0}^{T} \int_{at}^{t}
\E \big[ \Theta_{\frac{f(0)}{f(1)}} ( \lambda_s ) \big]
\, \dd s\, \dd t
\end{align*}
where we used the fact that $\phi$ is decreasing on $(0,1].$ The function $z \mapsto \Theta_{c}(z)$ is readily verified to be convex; therefore, by Jensen’s inequality and Lemma \ref{lem:martingale} 
\begin{displaymath}
    \delta_{\mathrm{Thinn}}(f) \ge a(1-a)\frac{T^{2}}{2} \Theta_{\frac{f(0)}{f(1)}} \Bigg( \frac{\E X}{T} \Bigg).
\end{displaymath}
\end{proof}

\begin{remark}
    The stability estimate of Theorem \ref{thm:main1.2} ultimately reduces to obtaining a lower bound on the quantity $\partial_{s} \E \phi (\lambda_{s}),$  which was previously studied in \cite{ARS} in the context of the stability of Wu's inequality. Consequently, our argument relies heavily on the approach developed therein.
\end{remark}

\begin{lemma}
     \label{lem:derivative}
    Let $X$ be a random variable with density $f$ with respect to $\pi_{T}.$ Assume that $f$ is bounded away from zero and infinity and $H(X|\pi_{T}) < \infty.$ Then, for every $a$ in $[0,1],$ we have 
    \begin{displaymath}
        \frac{d}{da} \big(  H  ( T_{a}(X) | \pi_{aT} ) \big) = T \E \phi( \lambda_{aT} ).
    \end{displaymath}
    Moreover, $a \mapsto H\big( T_{a}(X) | \pi_{aT}  \big)$ is convex.
\end{lemma}
\begin{proof} 
For every $t \in [0,T],$ the Poisson point process $N$ (and therefore $(X_{t})_{t \in [0,T]}$) almost surely has no jump at time $t.$ Consequently, $\lambda_{t}$ is almost surely continuous at $t,$ and the boundedness assumption implies the continuity of the map 
\begin{equation}
      \label{eq:continuity}
    t \mapsto \E \phi (\lambda_{t}).
\end{equation}
By Fubini's theorem we have
\begin{displaymath}
    H( X |\pi_{T} ) = \int_{0}^{T} \E \phi (\lambda_{t}) \dd t < \infty
\end{displaymath}
and as a result
\begin{displaymath}
    t \mapsto \E \phi (\lambda_{t})
\end{displaymath}
is integrable on $[0,T].$ By Proposition \ref{prosotion:3.3} for every $a \in [0,1],$ we have
\begin{equation}
    \label{eq:deriv}
    \frac{d}{da} \Big( H(T_{a}(X)|\pi_{aT}) \Big) = \frac{d}{da} \Big( \int_{0}^{aT} \E(\phi(\lambda_{s})) ds \Big) = T \E( \phi(\lambda_{aT})).
\end{equation}
The moreover part follows from the submartingale property of $(\phi(\lambda_{t}))_{t}.$ The monotonicity of the derivative of $a \mapsto H \big( T_{a}(X)  | \pi_{aT} \big)$ implies that the function is convex.
\end{proof}

\begin{remark}
    The convexity of the thinning operation was proved in \cite{Y}; more precisely, it follows from the proof of \cite[Lemma 1]{Y}. We return to the derivative of the thinning operation in Subsection \ref{sss:size_bias}.
\end{remark}

\begin{remark}
     \label{remark:convexity}
    Let $f$ be bounded away from zero and infinity. The convexity of 
    \begin{displaymath}
        a \mapsto H \big( T_{a} (X) | \pi_{aT} \big) 
    \end{displaymath}
    is stronger than Wu's inequality (see \cite{W}). This observation was already made by Yu in \cite[Theorem 4]{Y}. We adapt his argument to the present setting. For every $a \in (0,1)$ we have by convexity
    \begin{equation}
         \label{eq:limit_Yu}
        \frac{H(T_{a}(X) | \pi_{aT} )}{a} \le \frac{d}{da} \big( H(T_{a}(X) | \pi_{aT}) \big) = T \E \phi (\lambda_{aT} )
    \end{equation}
    and by taking $a \rightarrow 1^{-},$ \eqref{eq:limit_Yu} yields
    \begin{displaymath}
        H \big( X|\pi_{T} \big) \le T \E \phi(\lambda_{T}) = T \E_{\pi_{T}} \big[ \Psi (f,Df) \big] ,
    \end{displaymath}
    where we used \eqref{eq:calculation} and $X_{T} \sim f \dd \pi_{T}.$
\end{remark}

\subsection{Fisher Information of $a$--thinning}
\label{sss:Fisher_information}

\begin{remark}
    The Fisher Information defined in \eqref{eq:FI} can be expressed via the Poisson--F\"ollmer process as follows. Recall that
    \begin{displaymath}
        \lambda_{T} = \frac{f(X_{T} + 1)}{f(X_{T})}
    \end{displaymath}
    and by Lemma \ref{lem:process} $X_{T}  \sim f \dd \pi_{T}.$ Thus, we can write
    \begin{equation}
         \label{eq:FI_stoch}
        I(X | \pi_{T} ) = T \int_{\N} \Big( \frac{Df}{f} \Big)^{2} f \dd \pi_{T} = T \E \big( (\lambda_{T} - 1 )^{2} \big).
    \end{equation}
    Similarly, we may express the Fisher Information of the thinned random variable. Indeed, we recall    
    \begin{displaymath}
        \lambda_{aT} = \frac{P_{T(1-a)}f(X_{aT} + 1) }{P_{T(1-a)} f(X_{aT}) }
    \end{displaymath}
    and from Lemmas \ref{lem:process} and \ref{lem:law}, it follows that $T_{a}(X) \overset{d}{=} X_{aT},$ thus as above
    \begin{equation}
         \label{eq:FI_scaled_stoch}
        I(T_{a} (X) | \pi_{aT} ) =  aT \E \big( (\lambda_{aT} - 1 )^{2} \big).
    \end{equation}

\end{remark}
We now proceed to the proof of Theorem \ref{thm:FIsc}.
\begin{proof}
The cases $a=0$ and $a=1$ are trivial. We therefore assume $a \in (0,1).$ First assume that $f$ is separated from zero and infinity and then
\begin{displaymath}
    I(T_{a}(X) | \pi_{aT}) \overset{\eqref{eq:FI_scaled_stoch}}{=}  aT \E \big( (\lambda_{aT} - 1)^{2} \big) \le a T \E \big( ( \lambda_{T} - 1)^{2} \big) \overset{\eqref{eq:FI_stoch}}{=} a I(X | \pi_{T})
\end{displaymath}
due to the martingale property of $\lambda_{t}.$ In the general case, we may assume without loss of generality that $I(X |\pi_{T}) < \infty.$ Considering the truncated densities introduced in \eqref{eq:densities} and by Lemma \ref{lem:law}, $T_{a}(X) \sim P_{T(1-a)} f \dd \pi_{T}.$ Thus, we obtain
\begin{equation}
\label{eq:approximation_FI}
    \int_{\N} \frac{(DP_{T(1-a)}f_{k})^{2}}{P_{T(1-a)}f_{k}} \dd \pi_{aT} \le a \int_{\N} \frac{(Df_{k})^{2}}{f_{k}} \dd \pi_{T} .
\end{equation}
by the previous case. The right-hand side of \eqref{eq:approximation_FI} converges to $\int_{\N} \frac{(Df)^{2}}{f} \dd \pi_{T}$ by the dominated convergence theorem. To apply the dominated convergence theorem, we distinguish three cases. If $\frac{1}{k} \le f \le k,$ then $f_{k} = f,$ and consequently $\frac{(Df_{k})^{2}}{f_{k}} \le \frac{(Df)^{2}}{f}.$ The case $f < \frac{1}{k} = f_{k}$ can be treated similarly. Finally, if $f_{k} = k < f,$ then $\frac{(Df_{k})^{2}}{f_{k}} \le \frac{k^{2}}{f_{k}} = k < f,$ which is integrable under the measure $\pi_{T}.$ Moreover, using the bound $0 \le f_{k} \le f + 1$ and the integrability of $f,$ we obtain by dominated convergence theorem 
\begin{displaymath}
    P_{T(1-a)}f_{k} \longrightarrow P_{T(1-a)}f,
    \quad
    DP_{T(1-a)}f_{k} \longrightarrow DP_{T(1-a)}f,
\end{displaymath}
pointwise on $\N$. Therefore, the integrand on the left-hand side of \eqref{eq:approximation_FI} converges pointwise. Since it is nonnegative, Fatou's lemma allows us to conclude that
\begin{align*}
       \int_{\N} \frac{\big(DP_{T(1-a)}f\big)^{2}}{P_{T(1-a)}f} \dd \pi_{aT}
       &\le \liminf_{k \rightarrow \infty}
       \int_{\N} \frac{\big(DP_{T(1-a)}f_{k}\big)^{2}}{P_{T(1-a)}f_{k}} \dd \pi_{aT} \\
       &\overset{\eqref{eq:approximation_FI}}{\le} a \liminf_{k \rightarrow \infty}
       \int_{\N} \frac{(Df_{k})^{2}}{f_{k}} \dd \pi_{T} \\
     &= a \int_{\N} \frac{(Df)^{2}}{f} \dd \pi_{T}.
\end{align*}
\end{proof}

We are now ready for the proof of Theorem~\ref{thm:main1.4}.
\begin{proof}
From equations \eqref{eq:FI_stoch} and \eqref{eq:FI_scaled_stoch}, the deficit becomes
\begin{align*}
\delta_{\mathrm{FI}}(f)
&:= a I(\mu |  \pi_T) - I (T_a(\mu) | \pi_{aT}) \\
&= a T \Big( \E ( \lambda_{T} - 1)^{2} - \E ( \lambda_{aT} - 1)^{2}  \Big) = a T \Big( \E  \lambda_{T} ^{2} - \E  \lambda_{aT} ^{2}  \Big) 
\end{align*}
where we used the fact that $\E \lambda_{s}$ is constant; see Lemma \ref{lem:martingale}. Applying Lemma \ref{lem:Ito} for $h(x)=x^{2}$, we obtain
\begin{equation}
\label{eq:derivative_sq}
\begin{aligned}
\delta_{\mathrm{FI}}(f)
&= a T \int_{aT}^{T} \E \Big[ \lambda_s 
\big( DG(s,X_s) \big)^2 \Big] \dd s.
\end{aligned}
\end{equation}
Now, combining \eqref{eq:calc} and Proposition \ref{proposition:ARS}, we get
\begin{displaymath}
    1 + \frac{ DG(s,X_{s})}{G(s,X_{s})} = e^{D^{2} F(s,X_{s})} \le \frac{1}{1 + \frac{f(0)}{f(1)} \lambda_{s}} < 1.
\end{displaymath}
Recall that $\lambda_{s} = G (s,X_{s})$ and therefore, after rearranging terms,
\begin{displaymath}
    DG(s,X_{s}) \le \lambda_{s} \Bigg( \frac{-\frac{f(0)}{f(1)} \lambda_{s}}{1 + \frac{f(0)}{f(1)} \lambda_{s}} \Bigg) 
\end{displaymath}
which, since the right-hand side is negative, implies that
\begin{equation}
     \label{eq:ineq}
    DG(s,X_{s}) ^{2} \ge \lambda_{s}^{4} \Bigg( \frac{\frac{f(0)}{f(1)} }{1 + \frac{f(0)}{f(1)} \lambda_{s}} \Bigg)^{2}.
\end{equation}
Combining \eqref{eq:derivative_sq} with \eqref{eq:ineq} yields
\begin{align*}
\delta_{\mathrm{FI}}(f)
&= a T \int_{aT}^{T} \E \Big[ \lambda_s 
\big( DG(s,X_s) \big)^2 \Big] \dd s \\
&\ge a T \int_{aT}^{T} \E \big[ 
\Delta_{\frac{f(0)}{f(1)}} 
( \lambda_s )  \big] \, \dd s .
\end{align*}
Since $z \mapsto \Delta_{c}(z)$ is convex, as follows from the positivity of its second derivative, Jensen’s inequality and Lemma \ref{lem:martingale} imply
\begin{displaymath}
    \delta_{\mathrm{FI}}(f)\ge a(1-a)T^2 \Delta_{\frac{f(0)}{f(1)}} \Bigg( \frac{\E X}{T} \Bigg).
\end{displaymath}
\end{proof}

\begin{remark}
\label{remark:comparison}
Let us compare the above estimate with \eqref{eq:log_conc} of Harrem\"oes, Johnson and Kontoyiannis who proved it under the assumption that $f$ is log-concave. In view of Remark \ref{rem:definitions}, we set $T = \E X,$ and Theorem \ref{thm:main1.4} gives  
\begin{equation}
    \label{eq:ULC_estimate}
    I \big( T_{a} (X) | \pi_{aT} \big) \le a \Big( I(X | \pi_T) - T^{2} C^{2} \Big) + a^{2} T^{2}C^{2}
\end{equation}    
where $C := \frac{f(0)}{f(0) + f(1)}.$ Upon rearranging terms, we obtain $\forall a \in (0,1),$
\begin{align*}
a \Big( I(X | \pi_T) - T^{2} C^{2} \Big) + a^{2} T^{2}C^{2}
&< a^{2} I(X | \pi_T) \\
&\iff 0 < a(1-a)
\Big( T^{2}C^{2} - I(X | \pi_T) \Big) \\
&\iff I(X | \pi_T) < T^{2} C^{2}.
\end{align*}
so our bound is smaller if and only if
\begin{equation}
     \label{eq:comparison} 
    \E_{\pi_{T}} \Bigg( \frac{(Df)^{2}}{f}  \Bigg) < T C^{2} .
\end{equation}
We are now going to present a class of random variables
for which our estimate satisfies the condition \eqref{eq:comparison}. Let $t > 0$ and consider the family of measures
\begin{displaymath}
    \p (X =k) = \frac{t^{k}}{(k!)^{2}} \frac{1}{Z(t)},
\end{displaymath}
where $Z(t):= \sum_{k\ge 0} \frac{t^{k}}{(k!)^{2}}$ is the normalising constant. We introduce the functions $ \alpha(t) := \frac{Z(t)}{Z'(t)},$ $T(t):= \frac{t}{\alpha(t)}$ and we may sometimes suppress their dependence on $t$ for notational convenience. Then, we obtain $\E X = \frac{t Z'(t)}{Z(t)} = \frac{t}{\alpha(t)}= T(t) $ and the density function of $X$ with respect to $\pi_{T}$ is given by 
\begin{displaymath}
    f(k) = \frac{\p (X=k)}{\pi_{T}(k)} = \frac{\alpha^{k}}{k!} \frac{e^{T}}{Z(\alpha T)}  ,
\end{displaymath}
which satisfies the \textit{ultra log-concavity} definition. Note that $\p (X = k) = \frac{(k+1)^{2}}{t} \p (X = k + 1)$ and this implies
    \begin{displaymath}
          \E \Big( \frac{\alpha}{X + 1} \Big) = \sum_{k=0}^{\infty} \frac{\alpha}{k+1} \p(X=k) = \frac{\alpha}{ t} \sum_{k=0}^{\infty} (k+1) \p(X=k+1) = 1.
    \end{displaymath}
Therefore, the left-hand side of inequality \eqref{eq:comparison} can be expressed as  
     \begin{equation}
          \label{eq:left_hand_side}
           L(t):= \sum_{k=0}^{\infty} \Big( \frac{\alpha}{k+1} - 1\Big)^{2} \p(X=k) = \E \Big[ \Big( \frac{\alpha}{X+1}-1 \Big)^{2} \Big] = \alpha^{2} \E \Big[  \frac{1}{(X+1)^{2}} \Big] - 1.
    \end{equation}
Moreover,
    \begin{displaymath}
         \E \Big[  \frac{1}{(X+1)^{2}} \Big] = \frac{1}{Z(t)} \sum_{k=0}^{\infty} \frac{1}{(k+1)^{2}} \frac{t^{k}}{(k!)^{2}} = \frac{1}{t Z(t)} \sum_{k=0}^{\infty}  \frac{t^{k+1}}{((k+1)!)^{2}} = \frac{Z(t)-1}{t Z(t)}
    \end{displaymath}
and consequently, equation \eqref{eq:left_hand_side} is upper bounded by 
    \begin{equation}
          \label{eq:upper_bound_example}
          L(t) = \alpha^{2} \Big( \frac{Z(t)-1}{t Z(t)} \Big) - 1 < \frac{\alpha - T }{ T }.
    \end{equation}
The power series $Z(t)$ satisfies the following differential equation, obtained by differentiating the series term by term and comparing coefficients:
    \begin{displaymath}
        t Z''(t) + Z'(t) - Z(t) = 0, \quad \text{ with Z(0) = 1.} 
    \end{displaymath}
Then, the ratio $\alpha(t)= \frac{Z(t)}{Z'(t)}$ satisfies the first-order non linear Riccati equation
    \begin{equation}
        \label{eq:Riccati}
        \alpha'(t) = 1 + \frac{\alpha(t)}{t} - \frac{\alpha^{2}(t)}{t} , \quad \text{ with $\alpha$(0) = 1.}
    \end{equation}
For the comparison function $y(t):= \frac{1}{4} + \sqrt{t + \frac{9}{16}},$ with initial condition $y(0)=1$ and $y'> 1 +\frac{y}{t} - \frac{y^{2}}{t},$ we obtain $\alpha(t) < y(t) $ for all times $t > 0$ and 
\begin{displaymath}
    \alpha'(t) > 0 \iff \alpha^{2}(t) - \alpha (t) - t < 0 \iff \alpha (t) < \frac{1}{2} + \sqrt{t + \frac{1}{4}}. 
\end{displaymath}
Since $\alpha (t) < y(t) < \frac{1}{2} + \sqrt{t + \frac{1}{4}}$ for all $t > 0,$ the function $\alpha$ is strictly increasing. Moreover, $\alpha$ is unbounded from above since it satisfies \eqref{eq:Riccati}, hence it takes arbitrary values greater than one. Fix $t_{*}>0$ such that $\alpha (t_{*}) = 6$ and let $t \ge t_{*}.$ Hence, the inequality $ \alpha(t) < \frac{1}{4} + \sqrt{t + \frac{9}{16}}$ implies that $\big( \alpha (t) - \frac{1}{4} \big)^{2} < t +\frac{9}{16}$ and, in particular, yields the estimates $\alpha - T < \frac{\alpha + 1}{2 \alpha}$ and $\frac{2 \alpha^{2} - \alpha - 1}{2 \alpha} < T.$ Finally, from inequality \eqref{eq:upper_bound_example},   
    \begin{displaymath}
        L(t) < \frac{\alpha - T}{ T} < \frac{\alpha + 1}{2\alpha T} < \frac{\alpha + 1}{2 \alpha^{2} - \alpha - 1} \le \frac{2 \alpha^{2} - \alpha - 1}{2 \alpha (\alpha + 1)^{2} } < \frac{T}{(\alpha + 1)^{2}} := R(t)
    \end{displaymath}
where $R(t)$ is the right-hand side of inequality \eqref{eq:comparison} and the only inequality involving only $\alpha$ is a calculus estimate valid for $\alpha (t) \ge 6 .$ Here, the choice of the constant 6 could be slightly improved.

\end{remark}

\subsection{Law of Thin Numbers and Convergence Rates}
\label{sss:convrates}
For this subsection, we assume that $T=\E X$ and so by Lemma \ref{lem:martingale} we have that $\lambda_{0}=\E (\lambda_{t}) =1.$ First, we give a proof of the Law of Thin Numbers in Theorem ~\ref{thm:law}.

\begin{proof}
Let $X$ have density $f$ with respect to $\pi_{T}.$ We first consider the case in which $f$ is bounded away from zero and infinity. The data processing inequality (see \cite[Lemma 1.3.11]{CK}) yields, for $S(x_{1},...,x_{n}):= \sum_{i=1}^{n} x_{i},$ 
\begin{align}
\label{eq:datapr}
H \Big( \sum_{i=1}^{n} T_{1/n}(X_{i}) \Big| \pi_{T} \Big)
&= H \Big( S\big( T_{1/n}(X_{1}),\dots, T_{1/n}(X_{n}) \big)
\Big| S(\pi_{T/n},\dots,\pi_{T/n}) \Big) \\
&\le H \Big( (T_{1/n}(X_{1}),\dots,T_{1/n}(X_{n}))
\Big| (\pi_{T/n},\dots,\pi_{T/n}) \Big) \notag \\
&= \sum_{i=1}^{n} H\big( T_{1/n}(X_{i}) \mid \pi_{T/n} \big)
= n H\big( T_{1/n}(X)\mid\pi_{T/n} \big). \notag
\end{align}
where in the last two equalities we used the independence and identical distribution of the random variables $X_{i}.$ In view of \eqref{eq:datapr}, it remains to prove that
\begin{displaymath}
    \frac{H (T_{a}(X) | \pi_{aT} ) }{a } \longrightarrow 0, \quad a \rightarrow 0^{+}.
\end{displaymath}
Using Proposition \ref{prosotion:3.3} together with the submartingale property of $(\phi(\lambda_{t}))_{t},$ we obtain 
\begin{equation}
     \label{eq:before_limit}
    H (T_{a}(X) | \pi_{aT} ) = \E \int_{0}^{aT} \phi (\lambda_{t}) \dd t \le a T \E \phi( \lambda_{aT})  
\end{equation}
and therefore
\begin{displaymath}
     \lim_{a \rightarrow 0^{+}} \frac{H (T_{a}(X) | \pi_{aT} ) }{a } \le T \lim_{a \rightarrow 0^{+}} \E \phi( \lambda_{aT}) = T \phi (1) = 0.
\end{displaymath}
To treat the general case, we consider the densities $f_{k}$ defined in \eqref{eq:densities} and the corresponding random variables $Y_{k} \sim \frac{f_{k} \pi_{T}}{ \int_{\N} f_{k} \dd \pi_{T}}.$ Applying \eqref{eq:calculation} to the estimate \eqref{eq:before_limit} and using the definition of $\lambda_{aT},$ the previous case yields, for every $k\in\N,$
\begin{displaymath}
    H ( T_{a} (Y_{k}) | \pi_{aT} ) \le  \frac{aT}{\int_{\N} f_{k} \dd \pi_{T}}  \E_{\pi_{aT}} \Psi \big( P_{T(1-a)} f_{k} , DP_{T(1-a)} f_{k} \big).
\end{displaymath}
We also used the 1-homogeneity of $\Psi,$ which implies that the normalizing constant in the density of $Y_k$ factors out. By the dominated convergence theorem, using the bound $0 \le f_{k} \le f + 1$ and the integrability of $f,$ we obtain
\begin{displaymath}
    \int_{\N} f_{k} \dd \pi_{T} \longrightarrow \int_{N} f \dd \pi_{T}.
\end{displaymath}
Since $T_{a}(Y_{k})$ converges weakly to $T_{a}(X),$ combining the lower-semi continuity of relative entropy with Corollary \ref{prop:appendix_bound2} from the Appendix, we deduce
\begin{align*}
      H \big( T_{a}(X) | \pi_{aT} \big)
      &\le \liminf_{k \rightarrow \infty}
      H \big( T_{a}(Y_{k}) | \pi_{aT} \big) \\
      &\le \liminf_{k \rightarrow \infty}
      \frac{aT}{\int_{\N} f_{k} \dd \pi_{T}} \E_{\pi_{aT}} \Psi \big( P_{T(1-a)} f_{k} , DP_{T(1-a)} f_{k} \big) \\
      &\ = a T \E_{\pi_{aT}} \Psi \big( P_{T(1-a)} f , DP_{T(1-a)} f \big).
\end{align*}
To complete the proof, divide by $a$ and pass to the limit
\begin{displaymath}
     \lim_{a \rightarrow 0^{+}} \frac{H (T_{a}(X) | \pi_{aT} ) }{a } \le T \lim_{a \rightarrow 0^{+}} \E_{\pi_{aT}}  \Psi\big( P_{T(1-a)}f, DP_{T(1-a)}f \big) = T \Psi (1,0) = 0 .
\end{displaymath}
Passing to the limit is justified by Vitali's theorem and we also used the identity $P_{T}f(1) = 1,$ whenever $T = \E X.$ Indeed, by Corollary \ref{corollary:last_one}, for every $b < 1,$ the family
\begin{displaymath}
    \Big(  \Psi \big( P_{T(1-a)} f , DP_{T(1-a)} f \big) \Big)_{0 \le a \le b} 
\end{displaymath}
is uniformly integrable with respect to the corresponding family of measures $(\pi_{a})_{0 \le a \le b}.$  
\end{proof}
We will use the following subadditivity principle of Fisher information proved in \cite[Proposition 3]{HJK3}, in its special case of all random variables being identically distributed.
\begin{proposition}
     [Harrem\"oes, Johnson and Kontoyiannis \cite{HJK3}]
    Let $(X_{i})_{1 \le i \le n}$ be independent $\N$--valued random variables. Then, we have
    \begin{displaymath}
        I \Bigg( \sum_{i=1}^{n} X_{i} \Big| \pi_{\lambda} \Bigg) \le \sum_{i=1}^{n} \frac{\E X_{i}}{\lambda} I \Big( X_{i} \big| \pi_{\E X_{i} } \Big) 
    \end{displaymath}
    where $\lambda := \sum_{i=1}^{n} \E X_{i}  .$
\end{proposition}
In particular, if $X_{1}, X_{2}, . . . , X_{n}$ are i.i.d. random variables with mean $T$ then the above bound implies, 
\begin{equation}
\label{eq:subadd}
\begin{aligned}
H \Big( \sum_{i=1}^{n} T_{1/n}(X_{i}) \,\big|\, \pi_{T} \Big)
&\le I \Big( \sum_{i=1}^{n} T_{1/n}(X_{i}) \,\big|\, \pi_{T} \Big) \\
&\le \frac{1}{n}\sum_{i=1}^{n} I \Big( T_{1/n}(X_{i}) \,\big|\, \pi_{T/n} \Big)
= I \big( T_{1/n}(X) \,\big|\, \pi_{T/n} \big)
\end{aligned}
\end{equation}
where the first inequality is a consequence of the definition of Fisher information \eqref{eq:FI} and the Bobkov--Ledoux estimate \eqref{eq:wu}. Thus, we combine our estimates for the Fisher Information of $a$--thinning $I(T_{a}(X) | \pi_{aT})$ to obtain the following convergence rates for the Law of Thin Numbers. First, we give finite $n$ bounds:
\begin{corollary}
     Let $T = \E X >0$ and $X \sim f \dd \pi_{T}.$  
     \begin{enumerate}
         \item If $f : \N \rightarrow [0, \infty),$ then
         \begin{displaymath}
             H \Big( \sum_{i=1}^{n} T_{1/n}(X_{i}) |\pi_{T} \Big) \le \frac{I(X | \pi_{T})}{n}
         \end{displaymath}
         \item and if $f : \N \rightarrow (0,\infty) $ is ultra log-concave, then 
         \begin{displaymath}
             H \Big( \sum_{i=1}^{n} T_{1/n}(X_{i}) |\pi_{T} \Big) \le 
         \frac{1}{n} \Big( I(X | \pi_{T}) - T^{2} C^{2} \Big) + \frac{1}{n^{2}} T^{2} C^{2}
         \end{displaymath}
         where $C = \frac{f(0)}{f(1) + f(0)}.$
     \end{enumerate} 
\end{corollary}
\begin{proof}
    The proof follows by combining Theorems \ref{thm:FIsc}, \ref{thm:main1.4} with \eqref{eq:subadd}. 
\end{proof}

\begin{remark}
    \label{rem:general_comparison} Let us compare the general bound of the previous Corollary, obtained by Theorem \ref{thm:FIsc} with \eqref{eq:general_conc}. Our bound is smaller, for $n\ge 2,$ if
    \begin{equation}
        \label{eq:comparison2}
        \frac{I(X|\pi_{T})}{n} \le \frac{\Var(X)}{n T} < \frac{\Var(X)}{n T} + \frac{1}{n^{2}}
    \end{equation}
    and this is indeed the case even for the emblematic example of Bernoulli random variables whenever expectation is small enough. Consider $X,$ a Bernoulli random variable with parameter $0 < T < 1.$ Then, its density with respect to $\pi_{T}$ is given by $f(0)= e^{T}(1-T),f(1)=e^{T} $ and $f(k)=0$ for all $k \ge 2.$ A calculation shows $I(X | \pi_{T}) = \frac{T^{2}}{1 - T}$ and 
    \begin{displaymath}
        \frac{T^{2}}{1 - T} \le 1 - T = \frac{\Var(X)}{ T}
    \end{displaymath}
    if and only if $0 < T \le \frac{1}{2}.$ Hence, \eqref{eq:comparison2} is valid for $0 < T \le \frac{1}{2}.$
\end{remark}

For our next result, we will need the following lemma.

\begin{lemma}
\label{lem:Johnson}
    Let $X$ have density $f$ with respect to $\pi_{T}$ and $T = \E X > 0.$ Then 
    \begin{displaymath}
        \frac{P_{T-s}f(1)}{P_{T-s}f(0)} \le \frac{T}{T-s}
    \end{displaymath}
    for $s < T.$
\end{lemma}

\begin{proof}
We begin by rewriting the expressions appearing in the numerator and the denominator. First, for the denominator, we have
\begin{equation}
      \label{eq:denominator}
       \begin{aligned}
        P_{T-s}f(0)
        &= e^{s} \sum_{k=0}^{\infty} f(k) \Big( 1- \frac{s}{T} \Big)^{k} \pi_{T}(k) \\
       &= e^{s} \E \Big( 1- \frac{s}{T} \Big)^{X}
\end{aligned}
\end{equation}
and similarly for the numerator
\begin{equation}
      \label{eq:numerator}
       \begin{aligned}
        P_{T-s}f(1)
        &=  \sum_{k=0}^{\infty} k f(k) \frac{\pi_{T-s}(k)}{T-s}  \\
       &= e^{s} \sum_{k=0}^{\infty} \frac{k}{T} f(k) 
       \Big( 1- \frac{s}{T} \Big)^{ k -1 } \pi_{T}(k) = \frac{e^{s}}{T} \E X \Big( 1- \frac{s}{T} \Big)^{X-1}.
\end{aligned}
\end{equation}
Hence, by combining \eqref{eq:denominator},\eqref{eq:numerator} we have
\begin{displaymath}
    \frac{P_{T-s}f(1)}{P_{T-s}f(0)} = \frac{ \E X \big( 1- \frac{s}{T} \big)^{X-1} } {T \E  \big( 1- \frac{s}{T} \big)^{X} } = \frac{T}{T-s} \frac{ \E X \big( 1- \frac{s}{T} \big)^{X} } {T \E  \big( 1- \frac{s}{T} \big)^{X} }
\end{displaymath}
and the claim follows by upper bounding the second fraction from above by 1. Let $X'$ be an independent copy of $X$ and $r:= 1- \frac{s}{T} < 1.$ A standard calculation yields
\begin{displaymath}
    2 \Big[ \E \big( X r^{X} \big) - \E  X \E  \big( r^{X} \big) \Big] = \E \Big[ \big( X - X' \big) \big( r^{X} - r^{X'} \big) \Big] \le 0
\end{displaymath}
and the claim follows.    
\end{proof}

\begin{proposition}
     \label{prop:logconcave_estimate}
    Let $T = \E X  >0,$ and assume that $X \sim f \dd \pi_{T},$ where $f: \N \rightarrow (0, \infty) $ is log-concave. Then,
    \begin{displaymath}
        I ( T_{a}(X) | \pi_{aT} ) \le a^{2}T^{2} \frac{1}{1-a}
        +\frac{a^{3}T^{3}}{4} \Big( \frac{1}{1-a} \Big)^{2}.
    \end{displaymath}
    
\end{proposition}
\begin{proof}
By \cite[Proposition 2.2]{RS}, Poisson semigroup preserves log-concavity, that is $P_{t}f$ is log-concanve for all $t \ge 0.$ Hence, $DG(s,X_{s}) \le 0$ for all times. Moreover, by \cite[Corollary 2.4]{RS}, for all $t \in [0,T]$ and all $k \in \N,$ we have
\begin{displaymath}
    \frac{P_{T-t} f(k+1) }{P_{T-t}f(k) } \le \frac{f(1)}{f(0)}
\end{displaymath}
and consequently $\lambda_{s}$ and $ DG(s,X_{s})$ are bounded functions of $s$ and continuous except for finitely many points. By applying Lemma \ref{lem:Ito} for $h(x)=x^{2}$ and the Cauchy--Schwarz inequality, we obtain
\begin{displaymath}
    \partial_s \big( \E \lambda_{s}^{2} \big) = \E \Big[ \lambda_{s} \big( DG(s,X_{s}) \big)^{2} \Big] \le \sqrt{\E \lambda_{s}^{2}} \sqrt{\E \big( DG(s,X_{s}) \big)^{4}} 
\end{displaymath}
which implies by Gr\"onwall's estimate 
\begin{equation}
    \label{eq:after_CS}
    \sqrt{ \E \lambda_{t}^{2} } - \sqrt{ \lambda_{0}^{2}} \le \frac{1}{2} \int_{0}^{t} \sqrt{ \E \big( DG(s,X_{s}) \big)^{4} } ds \le \frac{1}{2} \int_{0}^{t} \sqrt{ \E \lambda_{s}^{4}} \dd s .
\end{equation}
In the last inequality of \eqref{eq:after_CS}, we used the identity $DG( s,X_{s}) = G(s,X_{s} + 1) - \lambda_{s}$ and the fact that log-concavity implies $DG(s,X_{s}) \le 0 .$ Moreover, since $P_{T-s}f$ is log-concave for all $s,$ we obtain
\begin{equation}
\label{eq:integral_1}
       \sqrt{ \E \lambda_{s}^{4}} \le \Bigg( \frac{P_{T-s}f(1)}{P_{T-s}f(0)} \Bigg)^{2} 
\end{equation}
and by Lemma \ref{lem:Johnson}
\begin{equation}
       \label{eq:integral_2}
        \begin{aligned}
        \int_{0}^{aT} \Bigg( \frac{P_{T-s}f(1)}{P_{T-s}f(0)} \Bigg)^{2} \dd s
        &\le \int_{0}^{aT} \frac{1}{\big(1 - \frac{s}{T} \big)^{2}} \dd s
        \overset{u=\frac{s}{T}}{=}
        T \int_{0}^{a} \frac{1}{(1-u)^{2}} \dd u \\
        &= T \Big(\frac{1}{1-a} -1 \Big) = \frac{aT}{1 - a}.
\end{aligned}
\end{equation}
Combining estimates \eqref{eq:integral_1},\eqref{eq:integral_2}, we bound the integral in \eqref{eq:after_CS} for $t = aT$
\begin{displaymath}
     \int_{0}^{aT} \sqrt{ \E \lambda_{s}^{4}} \dd s \le \frac{aT}{1 - a} 
\end{displaymath}
and the estimate \eqref{eq:after_CS} becomes
\begin{equation}
    \label{eq:integral_4}
     \sqrt{ \E \lambda_{aT}^{2} } - 1 \le \frac{1}{2} \int_{0}^{aT} \sqrt{ \E \lambda_{s}^{4}} \dd s \le \frac{aT}{2(1-a)}  .
\end{equation}
Hence,
\begin{align*}
        I\big(T_a(X)\mid\pi_{aT}\big)
        &=
        aT \big[\E \lambda_{aT}^{2}-1\big]
        \overset{\eqref{eq:integral_4}}{\le}
        aT\Bigg[ \Big(1 + \frac{aT}{2(1-a)} \Big)^{2} - 1 \Bigg] \\
        &=
        a^{2}T^{2} \frac{1}{1-a}
        +\frac{a^{3}T^{3}}{4} \Big( \frac{1}{1-a} \Big)^{2}.
\end{align*}
\end{proof}

\begin{remark}
Let $f: \N \rightarrow (0,\infty)$ ultra log-concave density of $X$ with respect to $\pi_{T}$ and assume $T = \E X.$ The bound in \eqref{eq:ULC_estimate} is smaller than that of Proposition \ref{prop:logconcave_estimate}, whenever $I (X | \pi_{T} ) \le T^{2} C^{2}.$ This is the same regime in which \eqref{eq:ULC_estimate} is smaller than the estimate \eqref{eq:log_conc}. Indeed, 
\begin{align*}
       a\big(I(X | \pi_T)-T^{2}C^{2}\big)+a^{2}T^{2}C^{2}
       &\le
        a^{2}T^{2}\Bigg(\frac{1}{1 + \frac{f(1)}{f(0)}}\Bigg)^{2}
       \le
        a^{2}T^{2} \frac{1}{1-a} \\
       &\le
       a^{2}T^{2} \frac{1}{1-a}
        +\frac{a^{3}T^{3}}{4} \Big( \frac{1}{1-a} \Big)^{2}.
\end{align*}

\end{remark}

\begin{corollary}
     Let $T = \E X>0$ and $X \sim f \dd \pi_{T}.$ Then, if $f: \N \rightarrow (0,\infty)$ log-concave
     \begin{displaymath}
         H \Big( \sum_{i=1}^{n} T_{1/n}(X_{i}) |\pi_{T} \Big) \le  \frac{1}{n(n-1)} T^{2}  +  \frac{1}{n (n-1)^{2} } \frac{T^{3}}{4}
     \end{displaymath}
     and for large enough $n,$ 
     \begin{displaymath}
          H \Big( \sum_{i=1}^{n} T_{1/n}(X_{i}) |\pi_{T} \Big) \le c  \frac{T^{2}}{n^{2}} 
     \end{displaymath}
     for $c>1$ constant arbitrarily close to 1.
\end{corollary}
\begin{proof}
    The proof follows by combining Proposition \ref{prop:logconcave_estimate} with \eqref{eq:subadd}. 
\end{proof}

Moreover, we also derive an asymptotic result that describes the limiting decay rate as $n \rightarrow \infty,$ for which it will be convenient to introduce the definition of the \textit{size--basing operation}. 
\newline

\textbf{Definition}: Given a discrete random variable taking values in $\N$ with $T=\E X>0,$ the \textit{size--basing} operation of $X$ is defined as
\begin{displaymath}
    \p \big( S(X) = k \big) =  \frac{k + 1}{T} \p (X = k+1), \quad \forall k \ge 0.
\end{displaymath}
We provide, for completeness, a short calculation of the expectation of the \textit{size--basing operation} that is well known in the information theoretic literature as well as another calculation of a quanity that will naturally appear in the following.
\begin{lemma}
    \label{lem:law_siaze_bias}
    Let $X $ have density $f$ with respect to $\pi_{T}$ and $T=\E X.$ Then, $S(X) \sim f(\cdot + 1) \dd \pi_{T}$ and $\E S(X) = \frac{\E X^{2}}{T} - 1.$ Moreover,
    \begin{displaymath}
        P_{T}f(2) = \frac{\E S(X)}{T} .
    \end{displaymath}
\end{lemma}
\begin{proof}
    Note that
    \begin{displaymath}
        \p( S(X)=k ) = \frac{k+1}{T} f(k + 1) \pi_{T} (k+1) = f(k+1) \pi_{T}(k)
    \end{displaymath}
    and its expectation is given by
    \begin{align*}
     \E S(X)
      &= \sum_{k=0}^{\infty} k f(k+1)\pi_{T}(k)
      = \sum_{k=0}^{\infty} (k+1)f(k+1)\pi_{T}(k) - 1 \\
      &= \frac{1}{T}\sum_{k=0}^{\infty} k^2 f(k)\pi_{T}(k) - 1
       = \frac{\E X^2}{T} - 1.
     \end{align*}
    To prove the moreover part, we observe that
    \begin{displaymath}
        P_{T}f(2) = \sum_{k=0}^{\infty} f(k+2) \pi_{T} (k) = \frac{1}{T} \sum_{k=0}^{\infty} k f(k+1) \pi_{T} (k) = \frac{\E S(X)}{T} .
    \end{displaymath}
\end{proof}

\begin{proposition}
     \label{prop:asymptotic}
    Let $T = \E X >0,$ and assume that $X \sim f \dd \pi_{T}.$ If $f$ is bounded away from zero and infinity, then
    \begin{displaymath}
        \limsup_{n \rightarrow \infty} \Bigg[ n^{2}  H \Big( \sum_{i=1}^{n} T_{1/n}(X_{i}) |\pi_{T} \Big)  \Bigg] \le  \Big( \frac{\Var(X)}{T } - 1 \Big)^{2}.
    \end{displaymath}
    
\end{proposition}

\begin{proof}
 Recall first that by \eqref{eq:subadd}, we have
\begin{equation}
     \label{eq:limits}
    \limsup_{n \rightarrow \infty} \Big( n^{2}H \Big( \sum_{i=1}^{n} T_{1/n}(X_{i}) |\pi_{T} \Big) \Big) \le \limsup_{n \rightarrow \infty} \Big( n^{2} I (T_{1/n}(X) | \pi_{T/n}) \Big).
\end{equation}
By \eqref{eq:FI_scaled_stoch} and since $\E \lambda_{t} = 1,$ we obtain 
\begin{displaymath}
    I( T_{a} (X) |\pi_{aT} ) = aT \big[ \E  \lambda_{aT}^{2} - 1 \big] .
\end{displaymath}
Dividing the above identity by $a^{2}$ and letting $a \rightarrow 0^{+},$ the derivative of $\E \lambda_{t}^{2}$ at 0 appears. Under our assumptions, the functions $s \mapsto \lambda_{s}$ and $s \mapsto DG(s,X_{s})$ are bounded and continuous except at finitely many points. Applying Lemma \ref{lem:Ito} with $h(x)=x^{2},$ we obtain
\begin{align*}
\lim_{a \rightarrow 0^{+}} \frac{I(T_{a}(X) | \pi_{aT} ) }{a^{2}}
&= T \lim_{a \rightarrow 0^{+}}
   \frac{\E \lambda_{aT}^{2} - 1}{a} \\[6pt]
&= T \frac{d}{da}\Big( \E\lambda_{aT}^{2} \Big)\Big|_{a=0^{+}}
   = T^{2} \lambda_{0} \big( DG(0,X_{0}) \big)^{2} \\[6pt]
&= T^{2} \big( P_{T}f(2) - 1 \big)^{2}
\end{align*}
where the last equality follows from the identities $P_{T}f(1) = 1$ whenever $T = \E X ,$ and $P_{T} f(0) =1.$ To conclude the proof, we apply Lemma \ref{lem:law_siaze_bias}, which yields 
\begin{displaymath}
    T^{2} \big( P_{T}f(2) - 1 \big)^{2} = T^{2} \Big( \frac{\E X^{2}}{T^{2}} - \frac{1}{T} - 1 \Big)^{2} = \Big( \frac{\Var(X)}{T} - 1 \Big)^{2}.
\end{displaymath}

\end{proof}

\subsection{Equality cases and Size-biasing operation}

\subsubsection{Equality Cases}
\label{sss:equality_cases}
We now characterize the equality cases for both the Thinning Lemma and the Fisher information comparison under $a$--thinning. For one implication, we only need the following straightforward identity for the relative entropy between Poisson measures with different parameters.
\begin{lemma}
    \label{lem:calc1}
    \begin{displaymath}
        H \big(\pi_{x} | \pi_{y} \big) = y H \big( \pi_{\frac{x}{y}} | \pi_{1} \big) = y \phi \Big( \frac{x}{y}  \Big), \quad \forall x,y>0.
    \end{displaymath}
\end{lemma}

\begin{proposition}
    (Equality cases of Thinning Lemma). Let $X$ be a random variable with density $f$ with respect to $\pi_{T}.$ Assume that $f : \N \rightarrow ( 0, \infty)$ satisfies $H( X | \pi_{T}) < \infty.$ Then, equality holds in the Thinning Lemma, i.e.,
    \begin{displaymath}
        H( T_{a}(X) | \pi_{aT}) = a H(X | \pi_{T} )
    \end{displaymath}
    for some $a \in (0,1)$ if and only if $f(k)=e^{\beta k + \gamma}$ for some $\beta ,\gamma \in \R$ (that is $X$ is Poisson distributed).
\end{proposition}
\begin{proof}
If $X$ is Poisson distributed, then by  \eqref{eq:thinn_poiss} and Lemma \ref{lem:calc1} one can verify the equality holds above. For the reverse implication, we apply the entropy representation formula of Proposition \ref{proposition:representation_ARS} from the Appendix; see \cite[Proposition 4.1]{ARS}. By Lemma \ref{lem:law}, $T_{a}(X) \sim P_{T(1-a)} f \dd \pi_{aT}.$ Therefore, applying Proposition \ref{proposition:representation_ARS} to the thinned random variable, we deduce
\begin{align*}
      H ( T_{a}(X) | \pi_{aT} )
      &= \int_{0}^{aT} \E_{\pi_{s}} \Psi \big( P_{aT-s} P_{T(1-a)} f , DP_{aT-s} P_{T(1-a)} f \big) \dd s \\
      &= \int_{0}^{aT} \E_{\pi_{s}} \Psi \big( P_{T-s} f , D P_{T-s} f \big) \dd s \\
      &\overset{s=at}{=} a \int_{0}^{T} \E_{\pi_{at}} \Psi \big( P_{T-at} f , D P_{T-at} f \big) \dd t.
\end{align*}
Hence, the equality can be written
\begin{displaymath}
    \label{eq:equality}
    \int_{0}^{T} \E_{\pi_{at}} \Psi \big( P_{T - at}f , DP_{T-at} f \big) \dd t = \int_{0}^{T} \E_{\pi_{t}} \Psi \big( P_{T - t}f , DP_{T-t} f \big) \dd t
\end{displaymath}
for some $a \in (0,1).$ The function $t \mapsto \E_{\pi_{t}} \Psi \big( P_{T - t}f , DP_{T-t} f \big)$ is nondecreasing. Indeed, we first note that $\Psi$ is convex on $\big\{ (u,v) \in \R^{2} : u >0 ,u + v > 0 \big\}.$ Therefore, by Jensen's inequality and the semigroup property, for every $0 \le s \le t \le T,$ we have
\begin{align*}
      \E_{\pi_{s}} \Psi \big( P_{T - s}f , DP_{T-s} f \big)
      &= \E_{\pi_{s}} \Psi \big( P_{t-s} P_{T - t}f , D P_{t-s} P_{T-t} f \big) \\
      &\overset{\eqref{eq:heateq}}{=} \E_{\pi_{s}} \Psi \big( P_{t-s} P_{T - t}f , P_{t-s} D P_{T - t} f \big) \\
      &\le \E_{\pi_{s}} P_{t-s} \Psi \big( P_{T - t}f , D P_{T - t} f \big)
      = \E_{\pi_{t}} \Psi \big( P_{T - t}f , D P_{T - t} f \big).
\end{align*}
Hence it follows that 
\begin{displaymath}
    \E_{\pi_{at}} \Psi \big( P_{T - at}f , DP_{T-at} f \big) = \E_{\pi_{t}} \Psi \big( P_{T - t}f , DP_{T-t} f \big)
\end{displaymath}
for almost every $t \in [0,T]$ and iterating the identity yields
\begin{equation}
\label{eq:iterate1}
\begin{aligned}
\E_{\pi_{t}} \Psi \big( P_{T-t}f , DP_{T-t}f \big)
&= \E_{\pi_{at}} \Psi \big( P_{T-at}f , DP_{T-at}f \big)
 = \cdots \\
&= \E_{\pi_{a^{n}t}} \Psi \big( P_{T-a^{n}t}f , DP_{T-a^{n}t}f \big).
\end{aligned}
\end{equation}
Since $a \in (0,1),$ we have $a^{n}t \rightarrow 0$ as $n \rightarrow \infty.$ Thus, by taking the limit in \eqref{eq:iterate1}
\begin{align*}
\E_{\pi_{t}} \Psi \big( P_{T-t}f , DP_{T-t}f \big)
&= \lim_{n \rightarrow \infty}
   \E_{\pi_{a^{n}t}} \Psi \big( P_{T-a^{n}t}f , DP_{T-a^{n}t}f \big) \\
&= \Psi \big( P_{T}f(0) , D P_{T}f(0) \big) = \Psi \big( 1 , P_{T}f(1) - 1 \big)
\end{align*}
and therefore $t \mapsto \E_{\pi_{t}} \Psi \big( P_{T-t}f , DP_{T-t}f \big)$ is constant for almost every $t \in [0,T).$ To pass to the limit, we applied Vitali's theorem. Indeed, by Corollary \ref{corollary:UI} from the Appendix, for every $t < T,$ the family
\begin{displaymath}
    \Big(  \Psi \big( P_{T - s} f , DP_{T - s} f \big) \Big)_{0 \le s \le t} 
\end{displaymath}
is uniformly integrable with respect to the corresponding family of measures $(\pi_{s})_{0 \le s \le t}.$ Fix some $b \in (0,1)$ such that 
\begin{equation}
     \label{eq:equality1}
     \E_{\pi_{bT}} \Psi \big( P_{T(1-b)}f , DP_{T(1-b)}f \big) = \Psi \big( 1 , P_{T}f(1) - 1 \big).
\end{equation}
Recalling \eqref{eq:calculation} and applying the semigroup property, the right hand side of \eqref{eq:equality1} is equal to $\phi \big( P_{T}f(1) \big) = \phi \big( P_{bT} (P_{T(1-b)}f) (1) \big).$ Hence, \eqref{eq:equality1} can be written 
\begin{align*}
\sum_{k=0}^{\infty} \phi \Bigg( \frac{P_{T(1-b)}f(k+1)}{P_{T(1-b)}f(k)} \Bigg)
P_{T(1-b)}f(k) \pi_{bT}(k)
&= \phi \big( P_{T}f(1) \big) {} \\
&\hspace{-2cm} =\phi\Bigg( \sum_{k=0}^{\infty}
\frac{P_{T(1-b)}f(k+1)}{P_{T(1-b)}f(k)}
P_{T(1-b)}f(k) \pi_{bT}(k) \Bigg).
\end{align*}
By Lemma \ref{lem:law}, $T_{b}(X) \sim P_{T(1-b)}f \dd \pi_{bT}.$ Since $\phi$ is strictly convex on $(0,\infty),$ by Jensen's equality cases, there exists $ c > 0 $ such that 
\begin{displaymath}
    P_{T(1-b)}f(k+1) = c P_{T(1-b)}f(k), \quad \forall k\in \N 
\end{displaymath}
hence recursively $P_{T(1-b)}f(k) = c^{k} P_{T(1-b)}f(0)$ for all $k \in \N.$ Thus, $P_{T(1-b)}f(0) = e^{bT(1-c)}$ and for all $k \in \N$
\begin{displaymath}
    \p \big( T_{b}(X) = k \big) =  c^{k} e^{bT(1-c)} \pi_{bT}(k) = \pi_{bcT}(k).
\end{displaymath}
For any $a \in (0,1)$ the map $X \mapsto T_{a}(X)$ is injective; see \cite[Proposition 5]{HJK2}. Hence $T_{b}(X) \overset{d}{=} T_{b}(\pi_{cT})$ implies that $X \sim \pi_{cT}$ and this yields the reverse implication.
\end{proof}
The same approach can be used to study the equality cases of Fisher information of the thinned random variable. For this, we need the following calculation.
\begin{lemma}
    \label{lem:calc2}
    \begin{displaymath}
        I \big( \pi_{x} | \pi_{y} \big) = \frac{(x-y)^{2}}{y}, \quad \forall x,y>0.
    \end{displaymath}
\end{lemma}

\begin{proposition}
    Let $X$ be a random variable with density $f$ with respect to $\pi_{T}.$ Assume that $f : \N \rightarrow ( 0, \infty)$ satisfies $\E_{\pi_{T}} \big( \frac{(Df)^{2}}{f} \big) < \infty.$ Then, equality holds in Theorem \ref{thm:FIsc}, namely
    \begin{displaymath}
        I( T_{a}(X) | \pi_{aT}) = a I(X | \pi_{T} )
    \end{displaymath}
    for every $a \in (0,1)$ if and only if $f(k)=e^{\beta k + \gamma}$ for some $\beta ,\gamma \in \R$ (that is $X$ is Poisson distributed).
\end{proposition}
\begin{proof}
If $X$ is Poisson distributed, then equality follows from \ref{eq:thinn_poiss} and Lemma \ref{lem:calc2}. Conversely, for the reverse implication, recall from Lemma \ref{lem:law} that $T_{a}(X) \sim P_{T(1-a)} f \dd \pi_{aT}.$ Since we assume that 
\begin{displaymath}
    \E_{\pi_{aT}} \Bigg( \frac{(D P_{T(1-a)}f)^{2}}{P_{T(1-a)}f} \Bigg) = \E_{\pi_{T}} \Bigg( \frac{(Df)^{2}}{f} \Bigg)
\end{displaymath}
for all $a \in (0,1),$ then by passing to the limit, we obtain
\begin{equation}
     \label{eq:equation2}
    \E_{\pi_{T}} \Bigg( \frac{(Df)^{2}}{f} \Bigg) = \lim_{a \rightarrow 0^{+}} \E_{\pi_{aT}} \Bigg( \frac{ \big(D P_{T(1-a)}f \big)^{2}}{P_{T(1-a)}f} \Bigg) =  \big( P_{T}f(1) - 1 \big)^{2}. 
\end{equation}    
To pass to the limit, we used Vitali's theorem. Note that $(x,y) \mapsto \frac{y^{2}}{x}$ is convex on $(0,\infty) \times \R,$ hence by Jensen's inequality,
\begin{displaymath}
    \frac{(DP_{T(1-a)}f)^{2}}{P_{T(1-a)f}} \le P_{T(1-a)} \Big( \frac{(Df)^{2}}{f} \Big).
\end{displaymath}
For any $M > 0,$ since $x \mapsto (x - M)_{+}$ is convex,
\begin{align*}
&\sup_{ 0 \le a \le 1 } \E_{\pi_{aT}} 
\Bigg(  P_{T(1-a)} \Big( \frac{(Df)^{2}}{f} \Big) - M \Bigg)_{+} \\
&\qquad \le \E_{\pi_{T}} 
\Bigg( \frac{(Df)^{2}}{f} - M \Bigg)_{+} \\
& \qquad \longrightarrow 0
\end{align*}
as $M \rightarrow \infty.$ The final limit vanishes because $\frac{(Df)^{2}}{f}$ is $\pi_{T}$--integrable by assumption. Thus, the smaller family $ \frac{(DP_{T(1-a)}f)^{2}}{P_{T(1-a)f}} $ is also uniformly integrable with respect to the corresponding measures. Consequently, \eqref{eq:equation2} may be written 
\begin{displaymath}
    \sum_{k=0}^{\infty}  \Bigg( \frac{f(k+1)}{f(k)} - 1 \Bigg)^{2} f(k) \pi_{T}(k) = \Bigg( \sum_{k=0}^{\infty} \Bigg(\frac{f(k+1)}{f(k)} - 1 \Bigg) f(k) \pi_{T}(k) \Bigg)^{2} 
\end{displaymath}
and since $x \mapsto x^{2}$ is strictly convex on $(0,\infty),$ by Jensen's equality cases there exists $ c > 0$ such that 
\begin{displaymath}
    f(k+1) = c f(k), \quad \forall k\in \N 
\end{displaymath}
and the result now follows.
\end{proof}

\subsubsection{Size-biasing operation} 
\label{sss:size_bias}

We conclude with a subsection that further explores the connection between our results and the existing literature. Yu in \cite{Y} proves that, under the assumption that $X$ has finite support and $T=\E X $, the following result holds
\begin{equation}
    \label{eq:Yau_deriv}
    \frac{d}{da} \big( H( T_{a}(X) | \pi_{aT} ) \big) = T H \big( T_{a} (S(X)) | T_{a}(X) \big).
\end{equation}
The author also asks whether this assumption can be relaxed. We provide a partial answer to this question here. 
\begin{corollary}
     \label{corol:derivative}
     Let $T=\E X$ and $X$ be a random variable with density $f$ with respect to $\pi_{T}.$ Assume that $f$ is bounded away from zero and infinity and $H(X|\pi_{T}) < \infty.$ Then, for every $a$ in $[0,1],$ we have
    \begin{displaymath}
        \frac{d}{da} \big(  H  ( T_{a}(X) | \pi_{aT} ) \big) = T H \big( T_{a} (S(X)) | T_{a}(X) \big) .
    \end{displaymath}
\end{corollary}

\begin{proof}
By Lemma \ref{lem:law_siaze_bias}, we have $S(X) \sim f( \cdot +1) \dd \pi_{T}$ and by Lemma \ref{lem:law}, we get $T_{a}(X) \sim P_{T(1-a)} f \dd \pi_{aT}$ and $T_{a}(S(X)) \sim P_{T(1-a)} f(\cdot + 1) \dd \pi_{aT}.$ Hence, we have 
\begin{displaymath}
    T_{a}(S(X))\sim P_{T(1-a)}f (\cdot+1) \dd \pi_{aT} = \frac{P_{T(1-a)}f (\cdot+1)}{P_{T(1-a)}f } P_{T(1-a)}f \dd \pi_{aT}
\end{displaymath}
and 
\begin{align*}
H \big(T_{a}(S(X)) | T_{a}(X) \big)
&= \sum_{k \in \N} \frac{P_{T(1-a)}f(k+1)}{P_{T(1-a)}f(k)}
\log \Bigg( \frac{P_{T(1-a)}f(k+1)}{P_{T(1-a)}f(k)} \Bigg)
\, P_{T(1-a)}f(k)\, \pi_{aT}(k) \\[0.4em]
&= \sum_{k \in \N } \phi \Bigg( \frac{P_{T(1-a)}f(k+1)}{P_{T(1-a)}f(k)} \Bigg)
\, P_{T(1-a)}f(k)\, \pi_{aT}(k) \\[0.4em]
&= \mathbb{E}\!\left(\phi(\lambda_{aT})\right).
\end{align*}
In the above calculation, we used 
\begin{align*}
\sum_{k \in \N} \Bigg( -\frac{P_{T(1-a)}f(k+1)}{P_{T(1-a)}f(k)} + 1 \Bigg)
P_{T(1-a)}f(k)\, \pi_{aT}(k) = 0 .
\end{align*}
So, Lemma \ref{lem:derivative} recovers Yu's formula \eqref{eq:Yau_deriv}.

\end{proof}

\begin{remark}
    As explained in Remark \ref{remark:existence}, Corollary \ref{corol:derivative} continues to hold under the assumption that $f$ is log-concave.
\end{remark}

\section*{Appendix}

The following well--known result can be found in \cite[Chapter 3]{CL}.
\begin{proposition} [Variational formula for entropy] 
       \label{prop:variational}
      Let $(\Omega, \mathcal{F},\mu)$ be a probability space, and let $f : \Omega \rightarrow (0,\infty) $ be a $\mu$--integrable function such that $\int_{\Omega} f \dd \mu = 1.$ Then 
      \begin{displaymath}
          \Ent_{ \mu } (f) = \sup_{g} \Big\{ \int f g \dd \mu - \log \Big( \int e^{g} \dd \mu \Big) \Big\}
      \end{displaymath}
      where the supremum is taken over measurable and bounded functions $g : \Omega \rightarrow \R.$

\end{proposition}

\begin{corollary}
     \label{corol:logarithm}
    Let $X$ have density $f$ with respect to $\pi_{T}.$ If $H(T_{a}(X) |\pi_{aT}) < \infty$ for some $a \in (0,1],$ then
    \begin{displaymath}
        a\E X = \E T_{a}(X) < \infty
    \end{displaymath}
    and
    \begin{displaymath}
         \E T_{a}(X) \log T_{a}(X) < \infty .
    \end{displaymath}
\end{corollary}

\begin{proof}
    Fix $a \in (0,1].$ By Lemma \ref{lem:law}, we have $T_{a}(X) \sim P_{T(1-a)} f \dd \pi_{aT}$ and it follows that $\Ent_{\pi_{aT}} (P_{T(1-a)} f) = H (T_{a}(X) | \pi_{aT} ).$ Apply the duality formula of Proposition \ref{prop:variational} for $g:\N \rightarrow \R $ given by $g(n) = \lambda n \log n,$ with $\lambda \in (0,1),$ $P_{T(1-a)} f$ instead of $f$ and measure $\pi_{aT}.$ Therefore,
\begin{displaymath}
    \E T_{a}(X) \log T_{a}(X) \le \frac{1}{\lambda} \Big[ H(T_{a}(X) | \pi_{aT} ) + \log \E_{\pi_{aT}} ( e^{g} ) \Big]
\end{displaymath}
and the last term is finite provided that $\lambda < 1 .$ Similarly, if we take $g(n) = n,$ we obtain from the moment generating function
\begin{displaymath}
    \E T_{a}(X) \le H (T_{a}(X)| \pi_{aT}) + aT(e - 1) < \infty.
\end{displaymath}
    Lastly, the equality $a\E X = \E T_{a}(X)$ follows by conditioning on $X$ and using the definition of the thinning operation.
\end{proof}

We will also use the following standard inequality, which can be found in \cite[Theorem 2.7.1]{CT}

\begin{proposition} [Log--sum inequality]
    \label{prop:log_sum} For non--negative numbers $(a_{i})_{i\ge 1}$ and $(b_{i})_{i \ge 1},$
    \begin{displaymath}
        \sum_{i \ge 1 } a_{i} \log \Bigg( \frac{\sum_{i \ge 1 } a_{i}}{\sum_{i \ge 1 } b_{i}} \Bigg) \le \sum_{i \ge 1 } a_{i} \log \Big( \frac{a_{i}}{b_{i}} \Big) 
    \end{displaymath}
    provided that the sums of $a_{i}$ and $b_{i}$ are finite.
\end{proposition}
We use here again the convention that $0 \log 0 = 0,$ $a \log \frac{a}{0} = \infty$ if $a > 0$ and
$0 \log \frac{0}{0} = 0.$ These follow from continuity. Now, we can prove the main result of the appendix.

\begin{proposition}
    \label{prop:appendix_bound1}
    Let $X$ have density $f$ with respect to $\pi_{T}$ and $H(T_{1-a} (X) | \pi_{(1-a)T} ) < \infty$ for some $a < 1.$ Then 
    \begin{displaymath}
        \E_{\pi_{aT}} \Psi \big( P_{T(1-a)} f , DP_{T(1-a)} f \big) < \infty.
    \end{displaymath}
    Moreover, if $H( X | \pi_{T} ) < \infty,$ then 
    \begin{displaymath}
        \E_{\pi_{t}} \Psi \big( P_{T (1-a) } f , DP_{T (1-a) } f \big) < \infty
    \end{displaymath}
    for all $ a < 1.$
\end{proposition}

\begin{proof}
Fix $a < 1$ and a straightforward calculation yields
\begin{align}
\Psi \big( P_{T(1-a)}f(k), D P_{T(1-a)}f(k) \big)
&= P_{T(1-a)}f(k+1)
\log \Bigg( \frac{P_{T(1-a)}f(k+1)}{P_{T(1-a)}f(k)} \Bigg) \notag \\
&\quad - D P_{T(1-a)}f(k).
\label{eq:psi}
\end{align}
Apply log-sum inequality to upper bound the first sum
\begin{align*}
& P_{T(1-a)}f(k+1) 
\log \Bigg( \frac{P_{T(1-a)}f(k+1)}{P_{T(1-a)}f(k)} \Bigg) \\
&= \sum_{n \in \N} f(k+1+n) \pi_{T(1-a)}(n) 
\log \Bigg( 
\frac{\sum_{n \in \N} f(k+1+n) \pi_{T(1-a)}(n)}
{\sum_{n \in \N} f(k+n) \pi_{T(1-a)}(n)} 
\Bigg) \\
&\le \sum_{n \in \N} f(k+1+n) \pi_{T(1-a)}(n) 
\log \Bigg( 
\frac{\sum_{n \in \N} f(k+1+n) \pi_{T(1-a)}(n)}
{\sum_{n \in \N} f(k+1+n) \pi_{T(1-a)}(n+1)} 
\Bigg) \\
&\le \sum_{n \in \N} f(k+1+n) \pi_{T(1-a)}(n) 
\log \Big( \frac{n+1}{T(1-a)} \Big) \\
&= \sum_{n \in \N}  \log (n+1) f(k+1+n) \pi_{T(1-a)}(n) 
- \log ( T(1-a) ) P_{T(1-a)}f(k+1) \\
&= \frac{1}{T(1-a)} \sum_{n \in \N} n \log (n) f(k+n) \pi_{T(1-a)}(n) 
- \log ( T(1-a) ) P_{T(1-a)}f(k+1).
\end{align*}
Hence, we obtain
\begin{align}
\Psi \big( P_{T(1-a) }f(k), D P_{T(1-a) }f(k) \big)
&\le P_{T(1-a)}f(k) \notag \\
&\quad + \frac{1}{T(1-a)} 
\sum_{n \in \N} n \log (n) f(k+n) \pi_{T(1-a)}(n) \notag \\
&\quad - \big( \log ( T(1-a) ) + 1\big) P_{T(1-a)}f(k+1).
\label{eq:calculation_entropy}
\end{align}
Recall from Lemma \ref{lem:law} that $T_{a}(X) \sim P_{T(1-a)}f \dd \pi_{aT}.$ Hence, taking expectation in the first term with respect to the measure $\pi_{aT},$ yields $\E_{\pi_{aT}} \big( P_{T(1-a)}f \big) = 1.$ For the third term, the semigroup property gives
\begin{displaymath}
    P_{aT} \big( P_{T(1-a)}f \big) (1) = P_{T} f(1) = \frac{\E X }{T} . 
\end{displaymath}
Finally, taking the expectation of the second term with respect to the measure $\pi_{aT},$ we obtain
\begin{align*}
\sum_{n \in \N } n \log (n) 
\Bigg( \sum_{k \in \N} f(k+n) \pi_{aT}(k) \Bigg) \pi_{T(1-a)}(n)
&= \sum_{n \in \N } n \log (n) P_{aT}(n) \pi_{T(1-a)}(n) \\
&= \E T_{1-a}(X) \log T_{1-a} (X).
\end{align*}
Putting them all together, we deduce 
\begin{align}
\E_{\pi_{aT}} \Psi \big( P_{T(1-a) }f, D P_{T(1-a) }f \big)
&\le 1 + \frac{1}{T(1-a)} 
\E T_{1-a}(X) \log T_{1-a} (X) \notag \\
&\quad - \big( \log ( T(1-a) ) + 1\big) \frac{\E X}{T}.
\label{eq:calculation_expectation_entropy}
\end{align}
The finiteness assertion follows from Corollary \ref{corol:logarithm} and estimate \eqref{eq:calculation_expectation_entropy}, while the moreover part is a consequence of the Thinning Lemma (Theorem \ref{thm:Yu}).
\end{proof}

\begin{corollary}
    \label{corollary:UI}
    Let $X$ have density $f$ with respect to $\pi_{T}.$ If $H(X|\pi_{T}) < \infty$ and $t<T,$ then
    \begin{displaymath}
        \Big(  \Psi \big( P_{T-s} f , DP_{T - s} f \big) \Big)_{0 \le s \le t} 
    \end{displaymath}
    is uniformly integrable with respect to the corresponding family of measures $(\pi_{s})_{0 \le s \le t}.$ 
\end{corollary}

\begin{proof}
Let $0 \le s \le t<T.$ By Jensen's inequality, we obtain $\Psi \big( P_{T - s}f, DP_{T-s} f \big) \le P_{t-s} \Psi \big( P_{T-t}f, DP_{T-t}f \big).$ For any $M > 0,$ since $x \mapsto (x - M)_{+}$ is convex,
\begin{align*}
&\sup_{ 0 \le s \le t } \E_{\pi_{s}} 
\big( P_{t-s} \Psi \big( P_{T-t}f, DP_{T-t}f \big) - M \big)_{+} \\
&\qquad \le \E_{\pi_{t}} 
\big( \Psi \big( P_{T-t}f, DP_{T-t}f \big) - M \big)_{+} \\
&\qquad \longrightarrow 0
\end{align*}
as $M \rightarrow \infty.$ The final limit vanishes because $\Psi \big( P_{T-t}f, DP_{T-t}f \big)$ is $\pi_{t}$--integrable by Proposition \ref{prop:appendix_bound1}. Thus, the smaller family $\big( \Psi \big( P_{T - s}f, DP_{T - s} f \big) \big)_{ 0 \le s \le t}$ is also uniformly integrable with respect to the corresponding family of measures $(\pi_{s})_{0 \le s \le t}.$ 
\end{proof}

\begin{corollary}
    \label{prop:appendix_bound2} Let $X$ have density $f$ with respect to $\pi_{T}$ and set 
    \begin{displaymath}
        f_{k}:= \min \Big \{ k , \max \Big\{ f , \frac{1}{k} \Big\} \Big\} .
    \end{displaymath}
    If $H(X|\pi_{T}) < \infty,$ then
    \begin{displaymath}
        \Big(  \Psi \big( P_{T-t} f_{k} , DP_{T - t} f_{k} \big) \Big)_{k \in \N}
    \end{displaymath}
    is uniformly integrable with respect to $\pi_{t},$ for every $t < T.$
\end{corollary}

\begin{proof}
Let $a=\frac{t}{T}<1.$ By \eqref{eq:calculation_entropy}, we have
\begin{align*}
\Psi \big( P_{T(1-a) }f(z), D P_{T(1-a) }f(z) \big)
&\le P_{T(1-a)}f(z) \notag \\
&\quad + \frac{1}{T(1-a)} 
\sum_{n \in \N} n \log (n) f(z+n) \pi_{T(1-a)}(n) \notag \\
&\quad - \big( \log ( T(1-a) ) + 1\big) P_{T(1-a)}f(z+1) \\
&\le P_{T(1-a)}f(z) \notag \\
&\quad + \frac{1}{T(1-a)} 
\sum_{n \in \N} n \log (n) f(z+n) \pi_{T(1-a)}(n) \notag \\
&\quad + \big( \big| \log ( T(1-a) ) \big| + 1\big) P_{T(1-a)}f(z+1) ,
\end{align*}
where in the last inequality we used triangle inequality. Apply the above formula for $f_{k}$ and use the bound $0 \le f_{k} \le f + 1,$ to obtain
\begin{align*}
\Psi \big( P_{T(1-a) }f_{k}(z), D P_{T(1-a) }f_{k}(z) \big)
&\le \big( P_{T(1-a)}f(z) + 1 \big) \\
&\quad + \frac{1}{T(1-a)}
\sum_{n \in \N} n \log (n) \big( f(z+n) + 1 \big) \pi_{T(1-a)}(n) \\
&\quad + \big( \big| \log ( T(1-a) ) \big| + 1\big)
\big( P_{T(1-a)}f(z + 1) + 1 \big).
\end{align*}
The function on the right hand side is $\pi_{aT}$--integrable. Indeed, by the Thinning Lemma (Theorem \ref{thm:Yu}), $H(T_{1-a}(X) | \pi_{(1-a)T} ) <\infty,$ which due to Corollary \ref{corol:logarithm} implies that both $\E X$ and $\E T_{1-a}(X) \log T_{1-a}(X) $ are finite. Finally, the term $\E_{\pi_{T(1-a)}} (g) $ is also finite for $g(n)= n \log(n).$
\end{proof}

\begin{corollary}
    \label{corollary:last_one}
    Let $X$ have density $f$ with respect to $\pi_{T}.$ If $H(X|\pi_{T}) < \infty$ and $T=\E X,$ then
    \begin{displaymath}
        H \big( T_{a} ( S(X)) | T_{a}(X) \big) = \E_{\pi_{aT}} \Psi \big( P_{T(1-a)}f , DP_{T(1-a)}f \big) < \infty
    \end{displaymath}
    for all $a < 1.$ Moreover, for any $a < 1 $
    \begin{displaymath}
        \Big(  \Psi \big( P_{T(1-b)} f , DP_{T(1-b)} f \big) \Big)_{0 \le b \le a} 
    \end{displaymath}
    is uniformly integrable with respect to the corresponding family of measures $(\pi_{b})_{0 \le b \le a}.$
\end{corollary}

\begin{proof}
We adapt the argument used in Proposition \ref{prop:appendix_bound1}. By the same calculation as in \eqref{eq:psi}, we have 
\begin{align}
\label{eq:psi_2}
      \Psi \big( P_{T(1-a)}f(k), D P_{T(1-a)}f(k) \big)
      &= P_{T(1-a)}f(k+1)
      \log \Bigg( \frac{P_{T(1-a)}f(k+1)}{P_{T(1-a)}f(k)} \Bigg) \notag \\
      &\quad - D P_{T(1-a)}f(k).
\end{align}
We note that the second term vanishes after taking expectation with respect to $\pi_{aT}.$ Indeed, $P_{T}f(1) = 1$ whenever $T=\E X.$ Consequently,
\begin{align*}
       \E_{\pi_{aT}} \big( D P_{T(1-a)} f \big)
        &= P_{aT} \big( P_{T(1-a)}f \big) (1) - P_{aT} \big( P_{T(1-a)}f \big) (0) \\
        &= P_{T} f(1) - P_{T} f(0) = 0.
\end{align*}
By Lemma \ref{lem:law_siaze_bias}, we have $S(X) \sim f( \cdot +1) \dd \pi_{T}$ and by Lemma \ref{lem:law}, we get $T_{a}(X) \sim P_{T(1-a)} f \dd \pi_{aT}$ and $T_{a}(S(X)) \sim P_{T(1-a)} f(\cdot + 1) \dd \pi_{aT}.$ Thus, it follows 
\begin{displaymath}
    T_{a}(S(X))\sim P_{T(1-a)}f (\cdot+1) \dd \pi_{aT} = \frac{P_{T(1-a)}f (\cdot+1)}{P_{T(1-a)}f } P_{T(1-a)}f \dd \pi_{aT} .
\end{displaymath}
By taking expectation with respect to $\pi_{aT}$ in \eqref{eq:psi_2}, we obtain
\begin{displaymath}
    H \big( T_{a} ( S(X)) | T_{a}(X) \big) = \E_{\pi_{aT}} \Psi \big( P_{T(1-a)}f , DP_{T(1-a)}f \big).
\end{displaymath}
The finiteness part follows from Proposition \ref{prop:appendix_bound1} and the moreover part follows by Corollary \ref{corollary:UI}.
\end{proof}

\begin{remark}
    \label{rem:remove_assumption}
    The following entropy representation formula stated in Proposition \ref{proposition:representation_ARS} for density functions is due to Aryan, Rivera, and Shenfeld; see \cite[Proposition 4.1]{ARS}. They prove the result under the additional assumption 
    \begin{displaymath}
        \E_{\pi_{T}} \Psi \big( f, Df \big) < \infty ,
    \end{displaymath}
    which is natural to assume in the context of equality cases of Wu's inequality. Adapting their proof to density functions, this assumption is used only to guarantee that the integrand on the right-hand side of \eqref{eq:variational_ARS} is always finite for $t < T.$ Indeed, using Jensen's inequality they establish the estimate
    \begin{displaymath}
        \E_{\pi_{t}} \Psi \big( P_{T - t} f , DP_{T - t} f \big)  < \infty.
    \end{displaymath}
    Proposition \ref{prop:appendix_bound1} allows us to remove this additional assumption. This extension is useful for our purposes, since the condition $\E_{\pi_{T}} \Psi \big( f, Df \big) < \infty$  is not natural in the study of equality cases for the Thinning Lemma; see Subsection \ref{sss:equality_cases}. 
\end{remark}

\begin{proposition}
     [Aryan, L\'opez-Rivera and Shenfeld \cite{ARS}]
    \label{proposition:representation_ARS}
    Let $X$ have density $f$ with respect to $\pi_{T}$ and $f: \N \rightarrow (0, \infty).$ If $H(X|\pi_{T}) < \infty,$ then
    \begin{equation}
        \label{eq:variational_ARS}
        H( X | \pi_{T} ) = \int_{0}^{T} \E_{\pi_{t}} \Psi \big( P_{T - t} f , DP_{T - t} f \big) \dd t.
    \end{equation}   
\end{proposition}

\end{document}